\documentclass[a4paper]{amsart}

\usepackage[hidelinks]{hyperref}
\usepackage[left=3.2cm,right=3.2cm,bottom=3cm,top=3cm]{geometry}

\newcommand{\myenumlabel}[1]{\textnormal{(\roman{#1})}}

\usepackage{aliascnt}

\makeatletter

\def\indsym#1#2{%
  \setbox0=\hbox{$\m@th#1x$}%
  \kern\wd0%
  \hbox to 0pt{\hss$\m@th#1\mid$\hbox to 0pt{$\m@th#1^{#2}$\hss}\hss}%
  \lower.9\ht0\hbox to 0pt{\hss$\m@th#1\smile$\hss}%
  \kern\wd0}
\newcommand{\ind}[1][]{\mathop{\mathpalette\indsym{#1}}}

\makeatother

\newcommand{\Mathbf}{\mathbf}

\numberwithin{equation}{section}

\newcommand{\myiffrench}[2]{#2}

\newcommand{\mynewthm}[3][]{%
  \newaliascnt{#2}{thmnum}%
  \newtheorem{#2}[#2]{#3}%
  \aliascntresetthe{#2}%
  \newtheorem*{#2*}{#3}%
  \expandafter\newcommand\csname #2autorefname\endcsname{#3}%
  \expandafter\renewcommand\csname the#2\endcsname{\thethmnum}%
}

\theoremstyle{plain}
\mynewthm[red]{thm}{\myiffrench{Théorème}{Theorem}}
\mynewthm{prp}{Proposition}
\mynewthm[purple]{lem}{\myiffrench{Lemme}{Lemma}}
\mynewthm[orange]{fct}{\myiffrench{Fait}{Fact}}
\mynewthm[purple!50!white]{cor}{\myiffrench{Corollaire}{Corollary}}

\theoremstyle{definition}
\mynewthm[green!80!black]{dfn}{\myiffrench{Définition}{Definition}}
\mynewthm{hyp}{\myiffrench{Hypothèse}{Hypothesis}}
\mynewthm{conv}{Convention}
\mynewthm{conj}{Conjecture}
\mynewthm{ntn}{Notation}
\mynewthm{cst}{Construction}

\theoremstyle{remark}
\mynewthm[yellow!90!black]{rmk}{\myiffrench{Remarque}{Remark}}
\mynewthm{qst}{Question}
\mynewthm{exc}{\myiffrench{Exercice}{Exercise}}
\mynewthm{exm}{\myiffrench{Exemple}{Example}}

\def\equationautorefname~#1\null{(#1)}
\def\itemautorefname~#1\null{#1}

\newcommand{\N}{{\Mathbf N}}
\newcommand{\R}{{\Mathbf R}}

\newcommand{\Q}{{\Mathbf Q}}

\newcommand{\bC}{{\mathbf C}}

\newcommand{\bL}{{\mathbf L}}
\newcommand{\bM}{{\mathbf M}}
\newcommand{\bN}{{\mathbf N}}

\newcommand{\bR}{{\mathbf R}}

\newcommand{\bU}{{\mathbf U}}

\newcommand{\mcA}{\mathcal{A}}
\newcommand{\mcB}{\mathcal{B}}

\newcommand{\mcD}{\mathcal{D}}

\newcommand{\mcH}{\mathcal{H}}

\newcommand{\mcK}{\mathcal{K}}

\newcommand{\mcQ}{\mathcal{Q}}

\newcommand{\eps}{\epsilon}

\newcommand{\eq}{^{\mathrm{eq}}}

\newcommand{\sub}{\subseteq}

\newcommand{\sdiff}{\triangle}
\newcommand{\set}[1]{\{#1\}}

\newcommand{\ip}[1]{\langle #1 \rangle}

\newcommand{\nm}[1]{\| #1 \|}

\newcommand{\cl}[2][]{\overline{#2}^{#1}}
\newcommand{\actson}{\curvearrowright}

\newcommand{\ud}{\,\mathrm{d}}

\newcommand{\MALG}{\mathrm{MALG}}

\DeclareMathOperator{\diam}{diam}

\DeclareMathOperator{\id}{id}
\DeclareMathOperator{\tp}{tp}
\DeclareMathOperator{\tS}{S}
\DeclareMathOperator{\acl}{acl}
\DeclareMathOperator{\dcl}{dcl}

\DeclareMathOperator{\clco}{\overline{co}}

\DeclareMathOperator{\Aut}{Aut}
\DeclareMathOperator{\Homeo}{Homeo}
\DeclareMathOperator{\Iso}{Iso}

\DeclareMathOperator{\Th}{Th}

\DeclareMathOperator{\RUCB}{RUCB}

\DeclareMathOperator{\WAP}{WAP}

\newcommand{\df}{\emph}

\renewcommand{\And}{\text{ and }}

\newcommand{\dslash}{\mathord\sslash}

\newcounter{cycprfcnt}
\newcounter{cycprffirst}
\newcommand{\cycprfpreamble}%
{%
  \setcounter{cycprfcnt}{1}
  \setcounter{cycprffirst}{0}
  \setlength{\itemindent}{0.5\leftmargin}%
  \setlength{\leftmargin}{0pt}%
  \newcommand{\cpcurr}{\myenumlabel{cycprfcnt}}%
  \newcommand{\cpnext}{\addtocounter{cycprfcnt}{1}\cpcurr}%
  \newcommand{\cpnum}[1]{\setcounter{cycprfcnt}{##1}\cpcurr}%
  \newcommand{\cpfirst}{\cpnum{1}}%
  \newcommand{\impnext}{\cpcurr{} $\Longrightarrow$ \cpnext.}%
  \newcommand{\impfirst}{\cpcurr{} $\Longrightarrow$ \cpfirst.}%
  \def\makelabel##1{\ifnum\value{cycprffirst}=0\hspace{-0.7\itemindent}\setcounter{cycprffirst}{1}\fi##1}%
}%

\newenvironment{cycprf}%
{\begin{list}{\impnext}%
  {\cycprfpreamble}}%
{\qedhere\end{list}}%

\newenvironment{cycprf*}%
{\begin{list}{\impnext}%
  {\cycprfpreamble}}%
{\end{list}}%

\title[WAP functions and stability]{Weakly almost periodic functions, model-theoretic stability, and minimality of topological groups}

\author{Ita\"\i{} Ben Yaacov}
\address{Universit\'e Claude Bernard -- Lyon 1 \\
  Institut Camille Jordan, CNRS UMR 5208 \\
  43 boulevard du 11 novembre 1918 \\
  69622 Villeurbanne \textsc{cedex} \\
  France}
\urladdr{\url{http://math.univ-lyon1.fr/~begnac/}}

\author{Todor Tsankov}
\address{Institut de Math{\'e}matiques de Jussieu--PRG \\
  Universit\'e Paris 7, case 7012 \\
  75205 Paris \textsc{cedex} 13 \\
  France}
\urladdr{\url{http://www.math.jussieu.fr/~todor/}}

\thanks{Research supported by the Institut Universitaire de France and ANR contract GrupoLoco (ANR-11-JS01-008).}


\usepackage[utf8]{inputenc}
\usepackage[T1]{fontenc}
\usepackage{amssymb}
\usepackage{stmaryrd}

\usepackage[osf]{mathpazo}

\usepackage{eucal}

\usepackage[initials, shortalphabetic]{amsrefs}

\begin{document}

\begin{abstract}
  We investigate the automorphism groups of $\aleph_0$-categorical structures and prove that they are exactly the Roelcke precompact Polish groups. We show that the theory of a structure is stable if and only if every Roelcke uniformly continuous function on the automorphism group is weakly almost periodic. Analysing the semigroup structure on the weakly almost periodic compactification, we show that continuous surjective homomorphisms from automorphism groups of stable $\aleph_0$-categorical structures to Hausdorff topological groups are open. We also produce some new WAP-trivial groups and calculate the WAP compactification in a number of examples.
\end{abstract}

\maketitle

\tableofcontents

\section{Introduction}
The main object of study in this paper are the automorphism groups of $\aleph_0$-categorical structures.
We recall that a structure is \emph{$\aleph_0$-categorical} if it is the unique countable (or separable, for metric structures) model of its first order theory.
It has been known for a while that there is a narrow correspondence between the model theory of an $\aleph_0$-categorical structure and the action of the automorphism group.
A classical theorem of Ryll-Nardzewski, Engeler, and Svenonius affirms that a countable (discrete) structure $\bM$ is $\aleph_0$-categorical if and only if the action $\Aut(\bM) \actson \bM$ is \emph{oligomorphic}, i.e., the diagonal action $\Aut(\bM) \actson \bM^n$ has only finitely many orbits for each $n$.
If that is the case, one can recover all model-theoretic information about $\bM$ from those actions.
This correspondence has created a new field at the interface of model theory, permutation group theory, combinatorics, and, more recently, computer science.
We invite the reader to consult the recent survey of Macpherson~\cite{Macpherson2011a} and the references therein for more information on this subject.

More recently, the correspondence described above has been generalized to the setting of \emph{continuous logic} and a theorem analogous to the Ryll-Nardzewski theorem, due to Henson, has been proved in this setting \cite{BenYaacov2007a}. Continuous $\aleph_0$-categorical structures include familiar examples from analysis such as the separable, infinite-dimensional Hilbert space, the measure algebra of a standard probability space, and separable atomless $L^p$ Banach lattices ($p < \infty$). Discrete structures, which we will call \emph{classical}, can also be considered as special cases of continuous structures.

The main examples of $\aleph_0$-categorical structures in both the classical and the continuous setting are the \emph{homogeneous structures}, i.e., structures $\bM$ for which all isomorphisms between finitely generated pieces of $\bM$ extend to full automorphisms of $\bM$. A homogeneous structure is $\aleph_0$-categorical if and only if the set of isomorphism classes of its $n$-generated substructures is finite (respectively, compact in a suitable topology). Familiar discrete examples include the countable dense linear order, the random graph, and the countable atomless Boolean algebra.

The automorphism groups of such structures are naturally endowed with the topology of pointwise convergence on the structure, which makes them Polish groups.
If $\bM$ is classical, its automorphism group is a permutation group, i.e., a closed subgroup of $S_\infty$, the group of all permutations of a countable discrete set.
Our first result, Theorem~\ref{thm:RoelckePreCompactGroup}, is a characterization of the Polish groups that occur as automorphism groups of $\aleph_0$-categorical structures: they are exactly the \emph{Roelcke precompact} Polish groups.
This generalizes a similar result about automorphism groups of classical structures from \cite{Tsankov2012}.

\begin{dfn}
  \label{dfn:RoelckePrecompact}
  A topological group $G$ is called \emph{Roelcke precompact} if for every neighbourhood $U \ni 1_G$, there exists a finite set $F \sub G$ such that $UFU = G$.
\end{dfn}

The notion of Roelcke precompactness was introduced by Roelcke and Dierolf~\cite{Roelcke1981} and later found a number of applications in the theory of topological groups, most notably through the work of Uspenskij~\cites{Uspenskii1998,Uspenskii2001,Uspenskii2008} and Megrelishvili~\cite{Megrelishvili2001}.

Because of the correspondence we mentioned earlier between $\aleph_0$-categorical structures and their automorphism groups, it is reasonable to expect that model-theoretic properties of $\bM$ will have natural counterparts as topological-group-theoretic properties of $\Aut(\bM)$ and vice versa. In fact, this correspondence can be made precise: by a theorem of Ahlbrandt and Ziegler~\cite{Ahlbrandt1986}, two classical $\aleph_0$-categorical structures have isomorphic automorphism groups if and only if they are bi-interpretable; therefore any property of $\bM$ invariant under bi-interpretability \emph{is} a property of the group $\Aut(\bM)$. The work on this paper started as an attempt to understand what corresponds on the group side to the one of the most important concepts studied in model theory, namely, \emph{stability}. It turns out that the absence of the order property, which characterizes stable formulas, can be written as invariance under exchanging limits, a condition that had appeared in the work of Grothendieck in the 1950s and that is equivalent to the weak compactness of a certain set of continuous functions. Using Grothendieck's result, we obtain the following (cf. Theorem~\ref{th:stable-equiv-r-equal-w}).
\begin{thm}
  \label{th:I:stability}
  Let $G$ be the automorphism group of an $\aleph_0$-categorical structure $\bM$. Then the following are equivalent:
  \begin{enumerate}
  \item $\Th(\bM)$ is stable;
  \item Every Roelcke uniformly continuous function on $G$ is weakly almost periodic.
  \end{enumerate}
\end{thm}
The Gelfand space of the C$^*$-algebra of weakly almost periodic functions on $G$ is the so-called WAP compactification of $G$ (denoted by $W(G)$), which carries the additional structure of a semitopological semigroup. It is possible to define the semigroup structure purely model-theoretically using the notion of stable independence. We carry this out in Section~\ref{sec:model-theor-viewp}.

It is sometimes possible to define a semigroup structure on the Roelcke compactification of $\Aut(\bM)$ even if the structure $\bM$ is not stable using notions of independence that do not come from stability (see \cite{Uspenskii2008}). However, in those cases, the semigroup is never semitopological and we do not consider such semigroups in this paper.

A property that has been studied quite extensively in topological group theory is that of minimality: a topological group $G$ is called \emph{minimal} if every bijective continuous homomorphism from $G$ to another Hausdorff topological group is a homeomorphism; $G$ is \emph{totally minimal} if every continuous surjective homomorphism to a Hausdorff topological group is open. We refer the reader to the recent survey by Dikranjan and Megrelishvili~\cite{Dikranjan2013} for more information about this concept. One of the main theorems in this paper is the following (cf. Theorem~\ref{th:min-topology}).
\begin{thm}
  \label{th:I:totally-minimal}
  Let $G$ be the automorphism group of an $\aleph_0$-categorical, stable structure, or equivalently, let $G$ be a Roelcke precompact Polish group such that every Roelcke uniformly continuous function on $G$ is weakly almost periodic. Then $G$ is totally minimal.
\end{thm}
There are a number of special cases of Theorem~\ref{th:I:totally-minimal} that were known before: for example, the infinite permutation group (Gaughan~\cite{Gaughan1967}), the unitary group (Stoyanov~\cite{Stojanov1984}; see also \cite{Uspenskii1998} for a different proof), and the automorphism group of a standard probability space (Glasner~\cite{Glasner2012a}). Some new examples for which minimality was not known before include automorphism group of $L^p$ lattices \cite{BenYaacov2011b} (or, which is the same, the group of measure-class-preserving automorphisms of a standard probability space), the automorphism groups of countably dimensional vector spaces over a finite field, and classical, $\aleph_0$-categorical, stable, non-$\aleph_0$-stable examples obtained via the Hrushovski construction (Wagner~\cite{Wagner1994}*{Example~5.3}).

Our proof is based on analysing the central idempotents in $W(G)$, an idea that goes back to Ruppert~\cite{Ruppert1990} and was first used in a setting similar to ours by Uspenskij~\cite{Uspenskii1998}.

We would like to emphasize that even though the proof of Theorem~\ref{th:I:totally-minimal} that we have included in this paper does not formally rely on model theory, the proof of the key Lemma~\ref{l:K-dense-subgroup} is a translation of a model-theoretic argument.

The weakly almost periodic functions on a group $G$ are exactly the matrix coefficients that arise from isometric representations of $G$ on reflexive Banach spaces, so, in a certain sense, understanding $W(G)$ amounts to understanding those representations. In particular, the triviality of $W(G)$ corresponds to the absence of non-trivial such representations. The first example of a group with a trivial WAP compactification, the group of orientation-preserving homeomorphisms of the interval $\Homeo^+([0, 1])$, was found by Megrelishvili~\cite{Megrelishvili2001}. A more detailed analysis of the WAP compactification of Roelcke precompact subgroups of $S_\infty$ leads to a new method for proving WAP triviality that applies to Megrelishvili's example but also yields a new one (see Corollary~\ref{c:wap-trivial}).
\begin{thm}
  \label{th:I:wap-trivial}
  Let $H$ be a Roelcke precompact subgroup of $S_\infty$ and let $\pi \colon H \to G$ be a homomorphism to another Polish group with a dense image. Suppose, moreover, that $G$ has no proper open subgroups. Then $G$ admits no non-trivial representations by isometries on a reflexive Banach space.
\end{thm}
The above theorem applies to the homeomorphism groups of some one-dimensional continua: for example, the interval and the Lelek fan (see Section~\ref{sec:minimality-topology}).

It is also possible to combine Theorem~\ref{th:I:totally-minimal} with automatic continuity to obtain results about uniqueness of group topologies. For example, using results of Hodges, Hodkinson, Lascar, and Shelah~\cite{Hodges1993a} and Kechris and Rosendal~\cite{Kechris2007a}, we have the following (see the end of Section~\ref{sec:minimality-topology} for a proof).
\begin{cor}
  \label{c:I:unique-topology}
  Let $G$ be the automorphism group of a classical, $\aleph_0$-categorical, $\aleph_0$-stable structure. Then $G$ admits a unique separable Hausdorff group topology.
\end{cor}
The conclusion of Corollary~\ref{c:I:unique-topology} also holds for the unitary group \cite{Tsankov2013}, the automorphism group of the measure algebra \cite{BenYaacov2013}, and the isometry group of the bounded Urysohn space (Sabok~\cite{Sabok2013p}).

Finally, our interpretation of WAP functions in terms of stable formulas allows us to calculate the WAP compactification of a number of specific groups, even in a non-stable situation. To our knowledge, these are the first explicit calculations of WAP compactifications in cases where the WAP compactification is both non-trivial and different from the Roelcke compactification.
We do this in Section~\ref{sec:exampl-wap-comp}.

This paper is organized as follows. In Section~\ref{sec:roelcke-prec-polish}, we collect some general facts about Roelcke precompact Polish groups and we prove that Roelcke precompact Polish groups are exactly the automorphism groups of $\aleph_0$-categorical structures. We also discuss a model-theoretic interpretation of the Roelcke compactification. In Section~\ref{sec:wap-compactification}, we turn to the WAP compactification and prove some basic facts about the semigroup $W(G)$. In Section~\ref{sec:minimality-topology}, we discuss the connection between WAP quotients of the group $G$ and the central idempotents in $W(G)$, and we prove Theorems~\ref{th:I:totally-minimal} and \ref{th:I:wap-trivial}. In Section~\ref{sec:model-theor-viewp}, we discuss in detail the model-theoretic interpretation of $W(G)$ and we prove Theorem~\ref{th:I:stability}. Finally, Section~\ref{sec:exampl-wap-comp} is devoted to examples.

\subsection*{Acknowledgements} We would like to thank Martin Hils for suggesting a piece of notation, Aleksandra Kwiatkowska for pointing out a mistake in an earlier version of the paper, Christian Rosendal for a useful remark, and S{\l}awomir Solecki for providing a reference. Part of the work on this paper was carried out in the Hausdorff Institute in Bonn during the trimester program \emph{Universality and Homogeneity} and we are grateful to the organizers and the Institute for the hospitality and the excellent working conditions they provided.

\section{Roelcke precompact Polish groups}
\label{sec:roelcke-prec-polish}

\subsection{A characterisation of Roelcke precompact Polish groups}
If $(X, d)$ is a metric space, we denote by $\Iso(X)$ the group of isometries of $X$. Equipped with the pointwise convergence topology, $\Iso(X)$ is naturally a topological group; if the metric space $(X, d)$ is Polish, $\Iso(X)$ is a Polish group. An \emph{isometric action} of a topological group $G$ on $X$ is just a continuous homomorphism $G \to \Iso(X)$ (the continuity of $G \rightarrow \Iso(X)$ is equivalent to the action map $G \times X \to X$ being jointly continuous). When $G$ acts on $X$ isometrically and $x \in X$, we let $[x]$ (or $[x]_G$, if there is a risk of ambiguity) denote the closure of the orbit of $x$. As the action of $G$ is isometric, the orbit closures form a partition of $X$.

If $I$ is a countable set, we equip $X^I$ with any metric inducing the product uniformity such that $d(x,y)$ depends solely on the individual distances $d(x_i,y_i)$, so that any isometric action on $X$ induces an isometric action on $X^I$ (such metrics always exist). When $I$ is finite, this will most often be the maximum metric $d(x,y) = \max_i d(x_i,y_i)$.

\begin{dfn}
  Let $X$ be a complete metric space, and let $G$ act on $X$ isometrically. We equip the set of orbit closures
  \[
  X \dslash G =  \bigl\{ [x]_G\colon x \in X \bigr\}
  \]
  with the distance induced from $X$
  \[
  d\bigl( [x], [y] \bigr) = \inf \, \bigl\{ d(u, v) : u \in [x], v \in [y] \bigr\}.
  \]
  We say that the action $G \curvearrowright X$ is \emph{approximately oligomorphic} if $X^n \dslash G$ is compact for all $n$.
\end{dfn}

The fact that $G$ acts on $X$ by isometries implies that $d$ on $X \dslash G$ is indeed a distance (that is, satisfies the triangle inequality and $d([x], [y]) = 0 \implies [x] = [y]$). It also coincides with the Hausdorff distance if we view equivalence classes $[x]$ as closed subsets of $X$.

For any Cauchy sequence of orbit closures, one can choose a Cauchy sequence of representatives, so the completeness of $X$ implies that of $X^n \dslash G$. In particular, $X^n \dslash G$ is compact if and only if it is totally bounded.

Next we check that if $X$ is complete, then $X^\bN \dslash G$ can be identified with the inverse limit $\varprojlim (X^n \dslash G)$. First, the projection maps $X^\N \to X^n$ are continuous and commute with the action of $G$, which means that we obtain well-defined maps $X^\N \dslash G \to X^n \dslash G$ and therefore a continuous map $\pi \colon X^\N \dslash G \to \varprojlim (X^n \dslash G)$. The maps $X^\N \dslash G \to X^n \dslash G$ are clearly surjective and now the completeness of $X^\N \dslash G$ implies that $\pi$ is surjective as well. Finally, it is easy to check that $\pi$ is a homeomorphism.

Therefore, $G \curvearrowright X$ is approximately oligomorphic iff $X^\bN \dslash G$ is compact.

\begin{dfn}
  \label{defn:GroupActionUniformities}
  Let $G$ be a topological group acting on the left on a set $X$. A symmetric neighbourhood $U$ of $1_G$ gives rise to an entourage $\bigl\{ (x,y) \in X^2 : x \in U \cdot y \bigr\}$, and these generate the \emph{right $G$-uniformity} on $X$.
  When $X$ is a topological space, the collection of bounded complex functions on $X$ that are continuous with respect to the topology on $X$ and right uniformly continuous with respect to the group action is denoted $\RUCB_G(X)$.

  In particular, the left action of $G$ on itself gives rise to the \emph{right uniformity} on $G$.
  Similarly, the right action of $G$ on itself gives rise to the \emph{left uniformity} on $G$.

  The greatest lower bound of the left and right uniformities on $G$ is called the \emph{Roelcke uniformity} (or, sometimes, the \emph{lower uniformity}). We say that $G$ is \emph{Roelcke precompact} if its completion with respect to the Roelcke uniformity is compact.
\end{dfn}

It is not difficult to check that if $G$ acts on $X$ continuously and isometrically, then the map $g \mapsto g \cdot x$ is left uniformly continuous on $G$.

On every topological group, the right and left uniformities are compatible with the topology. The Roelcke uniformity is generated by entourages of the form $UgU$, where $U \subseteq G$ is a symmetric neighbourhood of the identity, so it, too, is compatible with the topology. It follows that $G$ is Roelcke precompact if and only if for every non-empty open $U \subseteq G$ (equivalently, for every symmetric neighbourhood of $1_G$), there is a finite set $F \subseteq G$ such that $G = UFU$ (so \autoref{defn:GroupActionUniformities} agrees with \autoref{dfn:RoelckePrecompact}). A function on $G$ is Roelcke uniformly continuous if and only if it is both left and right uniformly continuous. Every Roelcke uniformly continuous function on a Roelcke precompact group is bounded.

Every metrizable topological group $G$ admits a left-invariant compatible distance and every such distance is compatible with the left uniformity. If $d_L$ is a left-invariant distance on $G$, then $d_R$ defined by $d_R(g, h) = d_L(g^{-1}, h^{-1})$ is a right-invariant distance (compatible with the right uniformity) and $d_{L \wedge R}$ defined by
\begin{equation}
  \label{eq:dRoelcke}
  d_{L \wedge R}(g, h) = \inf_{f \in G} \max \big(d_R(g, f), d_L(f, h) \big)
\end{equation}
is a distance compatible with the Roelcke uniformity of $G$. As metrizable compact spaces are second countable, this implies that  metrizable Roelcke precompact groups are separable.

Let $\widehat G_L = \widehat{(G,d_L)}$ denote the left completion of $G$. If $X$ is a metric space on which $G$ acts continuously by isometries, then the action $G \times X \rightarrow X$ extends by continuity to a map $\widehat G_L \times \widehat X \rightarrow \widehat X$, each $x \in \widehat G_L$ inducing an isometric embedding $\widehat X \rightarrow \widehat X$.
This applies in particular when $X = (G,d_L)$, so the group law on $G$ extends to a continuous semigroup law on $\widehat G_L$, and the map $\widehat G_L \times \widehat X \rightarrow \widehat X$ is a semigroup action by isometries. Similarly, \textit{mutatis mutandis}, for the right completion $\widehat G_R$.
The following is immediate.

\begin{lem}
  \label{lem:DenseTuple}
  Let $X$ be separable and complete, and let $\xi \in X^\bN$ be \emph{dense}, i.e., enumerate a dense subset of $X$.
  Let $G \leq \Iso(X)$ be a closed subgroup, and let $\Xi = [\xi] = \overline{G \cdot \xi}$.
  Then the distance $d_L(g,h) = d(g \cdot \xi, h \cdot \xi)$ is a compatible left-invariant distance on $G$, and the map $(\widehat G_L, d_L) \rightarrow \Xi$, $x \mapsto x \cdot \xi$ is an isometric bijection.
  If we identify $\Xi$ with $\widehat G_L$ in this manner, the diagonal action of $\widehat G_L$ on $X^\bN$ coincides with the semigroup law on $\widehat G_L$.
\end{lem}

Let $R(G) = \widehat G_L^2 \dslash G$ (where $G$ acts diagonally). Then the map $(G,d_{L \wedge R}) \rightarrow R(G)$ sending $g \mapsto [1_G,g] = [g^{-1}, 1_G]$ is isometric with a dense image and thus renders $R(G)$ the Roelcke completion of $G$, which we also denote by $\widehat G_{L \wedge R}$.  The involution $g \mapsto g^{-1}$ extends by continuity to bijections $\widehat G_L \rightarrow \widehat G_R$ and $\widehat G_{L \wedge R} \rightarrow \widehat G_{L \wedge R}$, which will be denoted by $x \mapsto x^*$.
In particular, elements of $\widehat G_R$ will be denoted $x^*$, where $x \in \widehat G_L$.
The group law on $G$ extends to continuous operations
\[
\widehat G_R \times \widehat G_L,\ \widehat G_R \times \widehat G_{L \wedge R},\ \widehat G_{L \wedge R} \times \widehat G_L
\longrightarrow \widehat G_{L \wedge R}.
\]
The first of these is given by $x^* y = [x,y]$, so every element of $\widehat G_{L \wedge R}$ can be written in this fashion, and everything is associative and respects the involution:
\begin{gather*}
  (x^* y)^* = y^* x, \qquad z^* (x^* y) = (xz)^* y = (z^* x^*)y, \qquad (x^* y)z = x^* (yz),
\end{gather*}
for $x,y,z \in \widehat G_L$. In particular, $\widehat G_{L \wedge R}$ is equipped with two commuting actions of $G$, one on the left, and one on the right. Consult \cite{Roelcke1981}*{Chapters 10,11} for more details on those constructions and complete proofs. Observe also that, by continuity, $x^*x = 1_G$ for all $x \in \widehat G_L$.

The following theorem gives a characterization of Roelcke precompact Polish groups in terms of approximately oligomorphic actions on metric spaces.
\begin{thm}
  \label{thm:RoelckePreCompactGroup}
  For a Polish group $G$, the following are equivalent:
  \begin{enumerate}
  \item $G$ is Roelcke precompact;
  \item Whenever $G$ acts continuously by isometries on complete metric spaces $X$ and $Y$, if both $X \dslash G$ and $Y \dslash G$ are compact then so is $(X \times Y) \dslash G$;
  \item Whenever $G$ acts continuously by isometries on a complete metric space $X$ and $X \dslash G$ is compact, the action is approximately oligomorphic;
  \item There exists a (complete and separable) metric space $X$ and a homeomorphic group embedding $G \hookrightarrow \Iso(X)$ such that the induced action $G \curvearrowright X$ is approximately oligomorphic.
  \end{enumerate}
\end{thm}
\begin{proof}
  \begin{cycprf}
  \item[\impnext]
    It will be enough to show that for any $\eps > 0$, $(X \times Y) \dslash G$ can be covered with finitely many balls of radius $2\eps$.
    Let us first cover $X \dslash G$ with a finite family of balls of radius $\eps$, say, centred at $[x]$ for $x \in X_0 \subseteq X$.
    Similarly, let us cover $Y \dslash G$ with $\eps$-balls centred at $[y]$ for $y \in Y_0 \subseteq Y$.
    Let $U \subseteq G$ be a symmetric neighbourhood of $1_G$ such that $\diam(U \cdot x) < \eps$ and $\diam(U \cdot y) < \eps$ for all $x \in X_0$, $y \in Y_0$, and let $F \subseteq G$ be a finite set such that $UFU = G$.
    Finally, let $W = \bigl\{ (x,f \cdot y)\colon x \in X_0, \, y \in Y_0, \, f \in F \bigr\} \subseteq X \times Y$.

    Consider now any $[u,v] \in (X \times Y) \dslash G$.
    First, there are $x \in X_0$, $y \in Y_0$ and $g_0,g_1 \in G$ such that $d(u,g_0 \cdot x) < \eps$ and $d(v,g_1 \cdot y) < \eps$.
    Second, there are $f \in F$ and $h_0, h_1 \in U$ such that $h_0^{-1} f h_1 = g_0^{-1} g_1$.
    Then
    \begin{gather*}
      [u,v] \sim [g_0 \cdot x,g_1 \cdot y] = [h_0  \cdot x,f h_1 \cdot y] \sim [x,f \cdot y] \in W,
    \end{gather*}
    where $\sim$ means distance $< \eps$.
    Thus $(X \times Y) \dslash G$ is covered by finitely many balls of radius $2\eps$, as desired.
  \item[\impnext]
    By induction, $X^n \dslash G$ is compact for all $n$.
  \item[\impnext]
    Let $\widehat G_L = \widehat{(G, d_L)}$ denote the left completion of $G$. Then the left action of $G$ on itself gives rise to a homeomorphic embedding $G \hookrightarrow \Iso(\widehat G_L)$ and $\widehat G_L \dslash G$ is a single point.
  \item[\impfirst]
    Replacing $X$ with $\widehat Y$, where $Y \subseteq X$ is an appropriate separable $G$-invariant subspace, we may assume that $X$ is separable and complete.
    Since $G \curvearrowright X$ is approximately oligomorphic, $X^\bN \dslash G \cong (X^\bN \times X^\bN) \dslash G$ is compact.
    Let $\xi \in X^\bN$ be dense and $\Xi = [\xi]$ as in \autoref{lem:DenseTuple}, in which case $\Xi^2 \dslash G \subseteq (X^\bN \times X^\bN) \dslash G$ is closed and therefore compact as well.
    In other words, $R(G) = \widehat G_L^2 \dslash G$ is compact, so $G$ is Roelcke precompact.
  \end{cycprf}
\end{proof}

\begin{rmk*}
  The direction (iii) $\Longrightarrow$ (i) of \autoref{thm:RoelckePreCompactGroup} is due to Rosendal~\cite{Rosendal2009}*{Theorem~5.2}.
  The full statement of \autoref{thm:RoelckePreCompactGroup} was observed independently by Rosendal and the first author of the present paper during a lecture of the second author in Texas in 2010, and a similar result appears in Rosendal~\cite{Rosendal2013}*{Proposition~1.22}.
\end{rmk*}

The next lemma is a translation of a simple amalgamation result from model theory. It will be useful in Section~\ref{sec:wap-compactification}.
\begin{lem}
  \label{lem:ComplexAmalgamation}
  Assume that $G$ is Roelcke precompact.
  For all $p \in \widehat G_{L \wedge R}$ and $x,y \in \widehat G_L$ such that $p x = p y$, there are $w,u,v \in \widehat G_L$ such that $w^* u = w^* v = p$ and $uy = vx$.
\end{lem}
\begin{proof}
  Say $p = z^* t$ for some $z, t \in \widehat G_L$, so that $px = py$ means that there is a sequence $h_n \in G$ such that $h_n z \rightarrow z$ and $h_n tx \rightarrow ty$.
  We may assume that the sequence $[z,h_n^{-1},1_G]$ converges to $[w,u_0,v_0]$ in $\widehat G_L^3 \dslash G$, i.e., there are $g_n \in G$ such that $g_n z \rightarrow w$, $g_n h_n^{-1} \rightarrow u_0$ and $g_n \rightarrow v_0$ in $\widehat G_L$.
  Letting $u = u_0 t$, $v = v_0 t$, we have $w^* u = w^* v = z^* t = p$ and $uy = \lim g_n h_n^{-1} h_n t x = vx$, as desired.
\end{proof}

\subsection{A model-theoretic interpretation of the Roelcke compactification}
\label{sec:roelcke-prec-polish-mt}

Theorem~\ref{thm:RoelckePreCompactGroup} above, together with classical results from model theory, gives abundant examples of Roelcke precompact Polish groups. More precisely, by combining a classical theorem of Ryll-Nardzewski, Engeler and Svenonius and its generalization to continuous logic \cite{BenYaacov2007a} with Theorem~\ref{thm:RoelckePreCompactGroup}, one obtains that Roelcke precompact Polish groups are exactly the automorphism groups of $\aleph_0$-categorical structures. We refer the reader to \cite{Hodges1993} for more details on this theorem in the classical situation and \cites{BenYaacov2007a,BenYaacov2008} for the continuous logic version and content ourselves with giving a general model-theoretic description of the Roelcke compactification and concrete calculations for a few examples.

A \emph{structure} $\bM$ is a complete metric space $(M, d)$ together with a set of predicates $\set{P_i : i \in I}$, where by a \emph{predicate}, we mean a bounded uniformly continuous function $P_i\colon M^{k_i} \rightarrow \bC$.
We call a structure $\bM$ \emph{classical} if all predicates, including the distance function, take only the values $0$ and $1$ (so $M$ is discrete).
The \emph{automorphism group} of $\bM$, denoted by $\Aut(\bM)$, is the group of all isometries of $(M, d)$ that also preserve the predicates, i.e., $P_i(g \cdot \bar a) = P_i(\bar a)$ for all $\bar a \in M^{k_i}$, $i \in I$. This is necessarily a closed subgroup of $\Iso(M)$, and therefore Polish when $M$ is separable.
We say that $\bM$ is \emph{$\aleph_0$-categorical} if it is separable and its first-order theory admits a unique separable model up to isomorphism.
By the Ryll-Nardzewski theorem, this holds if and only if the action $\Aut(\bM) \actson M$ is approximately oligomorphic.
Moreover, even though formulas are constructed syntactically from the predicates using continuous combinations and quantifiers, the Ryll-Nardzewski theorem tells us that when $\bM$ is $\aleph_0$-categorical, the (interpretations of) formulas on $M^\alpha$ are exactly the continuous $\Aut(\bM)$-invariant functions $M^\alpha \rightarrow \bC$, so all the logical information is contained in the action of $\Aut(\bM)$ on $M$.
Conversely, given any complete separable metric space $M$ and closed $G \leq \Iso(M)$, one can equip $M$ with predicates so as to obtain a structure $\bM$ with $\Aut(\bM) = G$.
For example, if $G$ is any Roelcke precompact Polish group then $G \curvearrowright \widehat G_L$ is approximately oligomorphic by \autoref{thm:RoelckePreCompactGroup}, so $\bM_G = (\widehat G_L,G)$ is $\aleph_0$-categorical (this construction is due to J.~Melleray).

Thus, for our purposes, we can define an $\aleph_0$-categorical structure as a pair $\bM = (M,G)$ where $M$ is a complete, separable metric space and $G \leq \Iso(M)$ is closed and acts approximately oligomorphically on $M$.
Classical $\aleph_0$-categorical structures correspond to such pairs where $M$ is equipped with the discrete $0$/$1$ distance.

For the rest of this subsection, let $\bM$ denote a fixed $\aleph_0$-categorical structure and let $G = \Aut(\bM)$; in particular, $G$ is Roelcke precompact.
Each $x \in \widehat G_L$ induces elementary embedding $x \colon \bM \rightarrow \bM$, and every elementary embedding arises uniquely in this fashion.

As per \autoref{lem:DenseTuple}, we identify $\widehat G_L$ with $\Xi = [\xi] \subseteq M^\bN$, where $\xi \in M^\bN$ is a dense sequence.
We may think of $\xi$ as an enumeration of $M$ (in fact, any tuple $\xi$ such that $M = \dcl(\xi)$ will suffice as well).
Accordingly, a point $x \in \Xi$ should be considered to enumerate the elementary substructure $x(\bM) \preceq \bM$.
A point $x^* y = [x,y] \in (\Xi \times \Xi) \dslash G = R(G)$ can be identified with $\tp(x,y)$, or, if one so wishes, with $\tp(x(M), y(M))$, and $G$ acts on $R(G)$ on either side by acting on the corresponding copy of $\bM$. An element $x^*y \in R(G)$ is an element of $G$ if the images $x(M)$ and $y(M)$ coincide (and then $x^{-1}y \colon M \to M$ is an automorphism of $\bM$).

Thus, a model-theorist may take a slightly different approach, and define directly
\begin{equation}
  \label{eq:2}
  R(G) = \set{\tp(x, y) : \tp(x) = \tp(y) = \tp(\xi)},
\end{equation}
i.e., morally speaking, the set of all possible ways to place two copies of $\bM$ one with respect to the other. This gives us a means to calculate $R(G)$ when $G = \Aut(\bM)$ for some familiar $\aleph_0$-categorical structures $\bM$. Below, we carry out the calculation in several examples.

\begin{exm}
  \label{ex:Sinfty}
  \textit{The full permutation group.} Let $M$ be a countable discrete set and $S_\infty$ denote the group of all permutations of $M$. Model-theoretically, $M$ is a countable structure in the empty language and $S_\infty$ is its automorphism group. By \eqref{eq:2}, $R(S_\infty)$ is the set of types of  pairs of embeddings $x, y \colon M \rightarrow M$; as the only element of the language is equality, the only information the type specifies is of the kind $x(a) = y(b)$ or $x(a) \neq y(b)$ for $a, b \in M$. We can therefore identify $\tp(x, y)$ with the partial bijection $x^{-1}y \colon M \to M$, whence
  \[
  R(S_\infty) = \set{\text{all partial bijections } M \to M},
  \]
  equipped with the topology inherited from $2^{M \times M}$. That $S_\infty$ is Roelcke precompact was first shown by Roelcke--Dierolf~\cite{Roelcke1981}; the compactification was calculated by Uspenskij~\cite{Uspenskii2002} and Glasner--Megrelishvili~\cite{Glasner2008}*{Section~12}.
\end{exm}

\begin{exm}
  \label{ex:Q}
  \textit{The dense linear ordering.} Let $(\Q, <)$ denote the set of rational numbers equipped with its natural linear order and let $\Aut(\Q)$ be its automorphism group. As before, $R(\Aut(\Q))$ is the set of types of pairs of embeddings $x, y \colon \Q \rightarrow \Q$. One way to visualize this is as the set of all linear orderings on $x(\Q) \cup y(\Q)$; this can be represented as a certain closed subset of $3^{\Q \times \Q}$, where for $\alpha \in 3^{\Q \times \Q}$, $\alpha(a, b)$ determines which of the three possibilities $x(a) < y(b)$, $x(a) = y(b)$, $x(a) > y(b)$ holds.
\end{exm}

\begin{exm}
  \label{ex:UH}
  \textit{The Hilbert space.} Let $\mcH$ be a separable, infinite-dimensional Hilbert space. The group of isomorphisms of $\mcH$ is $U(\mcH)$, its unitary group. The action $U(\mcH) \actson \mcH$ is not approximately oligomorphic because $\mcH \dslash U(\mcH) \cong \R^+$ is not compact. However, the action restricted to the unit sphere (a single orbit for $U(\mcH)$) is approximately oligomorphic and the action on the sphere determines the action on the whole space by scaling. As before, $R(U(\mcH))$ is the space of types of pairs of embeddings $x, y \colon \mcH \to \mcH$; such a type is determined by the values of the inner product $\ip{x(\xi), y(\eta)}$ for $\xi, \eta \in \mcH$, i.e., an element $p$ of $R(U(\mcH))$ is just a bilinear form $\ip{\cdot, \cdot}_p$ on $\mcH$ satisfying $|\ip{\xi, \eta}_p| \leq 1$ for $\xi, \eta$ in the unit sphere. Every such bilinear form defines a linear contraction $T_p$ on $\mcH$ by $\ip{T_p \xi, \eta} = \ip{\xi, \eta}_p$. We conclude that $R(U(\mcH))$ can be identified with the space $B(\mcH)_1$ of contractions equipped with the weak operator topology. The involution in this case is just the adjoint operation in $B(\mcH)_1$. The Roelcke compactification of $U(\mcH)$ was first computed by Uspenskij~\cite{Uspenskii1998}.
\end{exm}

\begin{exm}
  \label{ex:AutMu}
  \textit{The measure algebra.} Let $\MALG$ denote the measure algebra of a standard probability space $(X, \mu)$ (i.e., the collection of measurable subsets of $X$ modulo null sets). It is naturally a metric space with the distance $d(A, B) = \mu(A \sdiff B)$. Let $\Aut(\mu)$ denote its automorphism group. It is easy to check, using homogeneity, that the action $\Aut(\mu) \actson \MALG$ is approximately oligomorphic and $\Aut(\mu)$ is therefore Roelcke precompact. The Roelcke compactification of $\Aut(\mu)$ is the set of types of pairs of embeddings $\MALG \rightarrow \MALG$, or, dually, the set of types of pairs of measure-preserving maps $\pi_1, \pi_2 \colon X \to X$. The type of such a pair can be identified with the measure $(\pi_1 \times \pi_2)_*(\mu)$ on $X \times X$; we therefore obtain that
$R(\Aut(\mu))$ is the set of \emph{self-couplings} of $(X, \mu)$, i.e., all probability measures on $X \times X$ whose marginals are equal to $\mu$. The Roelcke compactification of $\Aut(\mu)$ was first computed by Glasner~\cite{Glasner2012a}; he described a different, but equivalent, representation.
\end{exm}

\begin{exm}
  \label{ex:IsoU1}
  \textit{The bounded Urysohn space.} Let $\bU_1$ denote the unique homogeneous Polish metric space of diameter bounded by $1$ universal for finite metric spaces of diameter bounded by $1$ and let $\Iso(\bU_1)$ be its isometry group. The type of a pair of embeddings $x, y \colon \bU_1 \to \bU_1$ is determined by the distances $d(x(a), y(b))$ for $a, b \in \bU_1$. Following Uspenskij~\cite{Uspenskii2008}, we see that a function $f \colon \bU_1 \times \bU_1 \to \R^+$ represents such a type if and only if it is \emph{bi-Kat\v{e}tov}, i.e., satisfies the conditions
\begin{align*}
  f(a,b) + f(a',b) &\geq d(a,a')  \quad \text{and} \\
  f(a,b) + d(a,a') &\geq f(a') \quad \text{for all } a, a', b \in \bU_1,
\end{align*}
as well as the symmetric ones for the second argument. Accordingly, $R(\Iso(\bU_1))$ can be identified with the space of all bi-Kat\v{e}tov functions on $\bU_1 \times \bU_1$ bounded by $1$, equipped with the pointwise convergence topology. This compactification was first identified in \cite{Uspenskii2008}.
\end{exm}

Note that the isometry group of the unbounded Urysohn space $\bU$ is not Roelcke precompact. This is a consequence of Theorem~\ref{thm:RoelckePreCompactGroup}: while $\bU \dslash \Iso(\bU)$ is a single point,
\[
\bU^2 \dslash \Iso(\bU) \cong \set{d(a, b) : a, b \in \bU} = \R^+
\]
is not compact.

\section{The WAP compactification}
\label{sec:wap-compactification}

Let $G$ be a Polish Roelcke precompact group. A function $f \in \RUCB(G)$ is called \emph{weakly almost periodic} if the orbit $G \cdot f$ is weakly precompact in the Banach space $\RUCB(G)$. It is well known that the space of weakly almost periodic functions $\WAP(G)$ is norm-closed and stable under multiplication (this is an easy consequence of Theorem~\ref{th:GrothendieckGroups} below), so $\WAP(G)$ is a commutative C$^*$-algebra. Let $W(G)$ denote its Gelfand space, so that $\WAP(G) \cong C(W(G))$. Every weakly almost periodic function is Roelcke uniformly continuous \cite{Ruppert1984}*{Chapter~III, Corollary~2.12}, whence we obtain a natural quotient map $R(G) \to W(G)$. In particular, $\WAP(G)$ is separable and $W(G)$ is metrizable. $W(G)$ is called the \emph{WAP compactification} of $G$ but note that this compactification (as opposed to the Roelcke compactification) is not always faithful, for example, it can be trivial (see Section~\ref{sec:minimality-topology}). Despite the fact that the compactification map $G \to W(G)$ is not always injective, we will often suppress it in our notation, i.e., consider elements $g \in G$ also as elements of $W(G)$.

One of the main facts about weakly almost periodic functions is the following theorem of Grothendieck~\cite{Grothendieck1952}.
\begin{thm}[Grothendieck]
  \label{th:GrothendieckGroups}
  A function $f \in \RUCB(G)$ is weakly almost periodic if and only if for all sequences $\set{g_n}_n, \set{h_m}_m \sub G$,
  \begin{equation}
    \label{eq:GrothendieckGroups}
    \lim_n \lim_m f(g_n h_m) = \lim_m \lim_n f(g_n h_m),
  \end{equation}
  whenever both limits exists.
\end{thm}
Using Theorem~\ref{th:GrothendieckGroups}, it is easy to define a semigroup law on $W(G)$: if $p, q \in W(G)$ with $p = \lim_n g_n$ and $q = \lim_m h_m$, $\set{g_n}_n, \set{h_m}_m \sub G$, we define $pq$ as $\lim_n \lim_m g_n h_m$. This multiplication is associative and continuous in each variable but, in general, not as a function of two variables, i.e., $W(G)$ has the structure of a \emph{semitopological semigroup}. The involution $p \mapsto p^*$, which we defined for the Roelcke compactification, descends naturally to $W(G)$, where it is continuous and compatible with the multiplication: $(pq)^* = q^* p^*$.

The compactification $W(G)$ has the following universal property: if $S$ is a compact semitopological semigroup and $\pi \colon G \to S$ is a continuous homomorphism, then $\pi$ extends to a homomorphism $W \to S$. This is because for every continuous function $f \in C(S)$, the function $f \circ \pi$ is WAP on $G$.

From now on, we will write $W$ for $W(G)$. As the map $G \to W$ is left uniformly continuous, it extends to a map $\widehat G_L \to W$. Similarly, we also have a map $\widehat G_R = \widehat G_L^* \to W$. From our description of the Roelcke compactification of $G$ and the fact that the map $R(G) \to W$ is surjective, it follows that every element of $W$ can be written as $x^*y$ for some $x, y \in \widehat G_L$.

Next we define the following partial preorder on $W$
\[
p \leq_L q \iff Wp \sub Wq
\]
and the corresponding equivalence relation
\[
p \equiv_L q \iff p \leq_L q \And q \leq_L p.
\]
There are an analogous preorder and equivalence relation if one considers right ideals in $W$ instead of left.
The two are exchanged by the involution.
The equivalence relation $\equiv_L$ is known as one of \df{Green's relations} in the semigroup theory literature.

Observe that for $x \in \widehat G_L$ and $p \in W$, $xp \equiv_L p$.

Below we will make use of the following important joint continuity theorem (\cite{Ruppert1984}*{Chapter~II, Theorem~3.6}).
\begin{thm}  \label{th:joint-cont}
  Let $S$ be a compact metrizable semitopological semigroup with a dense subgroup and let $p \in S$. Then the multiplication $S \times Sp \to S$ is jointly continuous at $(s, p)$ for every $s \in S$.
\end{thm}
In the next lemma, we collect several simple consequences of this theorem that we will use.
\begin{lem} \label{l:jcont-conseq}
  Let $p, q \in W$.
  \begin{enumerate}
  \item \label{i:l:jc:1} $q p \equiv_L p$ if and only if $q p \geq_L p$ if and only if $p = q^* q p$.
  \item \label{i:l:jc:3} Assume that $q = x^* y$, where $x,y \in \widehat G_L$. Then $p = qp$ if and only if $xp = yp$.
  \end{enumerate}
\end{lem}
\begin{proof}
  \ref{i:l:jc:1}. Assume $pq \geq_L p$.
  Let $g_n \to q$ with $g_n \in G$. Then $g_n p \to qp$, $g_n p \in Wqp$ and $p = g_n^{-1}g_np \to q^*qp$ by Theorem~\ref{th:joint-cont}.
  The rest is clear.

  \ref{i:l:jc:3}. If $p = qp$ then $x^*yp \equiv_L yp$.
  Applying \ref{i:l:jc:1} we obtain $xp = xx^*yp = yp$.
\end{proof}

An element $e \in W$ is called an \emph{idempotent} if $ee = e$. The following is a well-known fact about idempotents whose proof we include for completeness.
\begin{lem}
  \label{l:idempotents}
  For every idempotent $e \in W$, $e^* = e$;
\end{lem}
\begin{proof}
  We have $We = Wee$, so, by \autoref{l:jcont-conseq}, $e = e^*ee = e^*e$, and, applying the involution to both sides, $e^* = e^*e$.
\end{proof}

If $e$ is an idempotent, recall that the \emph{maximal group belonging to $e$} is the set
\[
H(e) = \set{p \in W : pe = ep = p \And \exists q, r \in W \ pq = rp = e}.
\]

\begin{lem}
  \label{l:maximal-group}
  \begin{enumerate}
  \item \label{i:l:inv:1} If $p \in H(e)$, then $p^*$ is both a left inverse and a right inverse of $p$, i.e., $pp^* = p^*p = e$.
  \item \label{i:l:inv:2} $H(e)$ is a $G_\delta$ subset of $W$ and therefore a Polish group.
  \end{enumerate}
\end{lem}
\begin{proof}
  \ref{i:l:inv:1}. This can be found in \cite{Ruppert1984}*{Chapter~III, Corollary~1.8} but the proof is short enough to include here.

  Suppose that $rp = e$. Then $e \equiv_L pe$ and by Lemma~\ref{l:jcont-conseq} \ref{i:l:jc:1}, $e = p^*pe = p^*p$. The argument that $pp^* = e$ is similar.

  \ref{i:l:inv:2}. Let $S = \cl{H(e)}$ and observe that by \ref{i:l:inv:1}, $S = S^*$. Let $\set{V_n}_n$ be a basis of open neighbourhoods of $e$ in $S$ satisfying $V_n^* = V_n$ and let $U_n = \set{q \in S : Sq \cap V_n \neq \emptyset}$. Note that as $U_n = \bigcup_{s \in S} \set{q : sq \in V_n}$, each $U_n$ is open in $S$. We claim that $H(e) = \bigcap_n (U_n \cap U_n^*)$. The $\sub$ inclusion being obvious, we check the other. Let $s_n \in S$ be such that $s_n q \in V_n$, so that $s_n q \to e$. By compactness, we may assume that $s_n$ converges to some $s \in S$, and $s_n q \to sq$, showing that $s$ is a left inverse of $q$. By a symmetric argument, $q$ also has a right inverse and is therefore an element of $H(e)$.

  Finally, multiplication is jointly continuous on $H(e)$ by Theorem~\ref{th:joint-cont}.
\end{proof}

The following lemma is a consequence of the proof of \cite{Uspenskii1998}*{Theorem~3.2} but we provide a proof for completeness.
\begin{lem}
  \label{l:LeastIdempotent}
  Let $S \subseteq W$ be a closed subsemigroup such that $S^* = S$. Then $S$ has a least $\equiv_L$-class and this class contains an idempotent, which is the unique $\equiv_L$-least idempotent in $S$.
\end{lem}
\begin{proof}
  By a combination of compactness and Zorn's Lemma, there exists $p \in S$ which is $\leq_L$-minimal. It follows that $p^* p \equiv_L p$, whereby $p = p p^* p$, so $e = p^* p$ is an idempotent which is $\leq_L$-minimal in $S$.

  Now suppose that $f \in S$ is another idempotent element which is $\leq_L$-minimal. Then $ef \equiv_L f$, so $f = e^* e f = ef$, and similarly $e = fe$. But then $e = e^* = e^* f^* = ef = f$.
\end{proof}

Say that an idempotent $e \in W$ is \emph{central} if $ew = we$ for all $w \in W$.
\begin{lem}
  \label{l:central-idemp}
  Suppose that $G$ is a Roelcke precompact Polish group that satisfies $R(G) = W(G)$.
  Let $e \in W(G)$ be a central idempotent, and let $K = \{p \in W(G) : pe = e\}$.
  If $p = x^* y \in K$ with $x, y \in \widehat G_L$, then there exists $q \in K$ satisfying $q x = y$.
\end{lem}
\begin{proof}
  By Lemma~\ref{l:jcont-conseq} \ref{i:l:jc:3}, $x^* y e = e$ implies $ex = xe = ye = ey$.
  Using Lemma~\ref{lem:ComplexAmalgamation}, we obtain $w,u,v \in \widehat G_L$ such that $w^* u = w^* v = e$ and $uy = vx$.
  Letting $q = u^* v$, we have $qx = u^*vx = u^*uy = y$.
  Also, $w^* v e = e$, so $ve = we$ by Lemma~\ref{l:jcont-conseq} \ref{i:l:jc:3}, and $q e = u^* v e = u^* w e = e$.
\end{proof}

\begin{lem}
  \label{lem:ApproximatelyInvertible}
  Let $G$ be a Roelcke precompact Polish group, $R = R(G)$.
  For an open set $1 \in V \subseteq R$, let $U_V = \{ q \in R : Gq \cap V \neq \emptyset \}$.
  Then $\bigcap_V U_V = \widehat G_L$ and $\bigcap_V (U_V \cap U_V^*) = G$.
\end{lem}
\begin{proof}
  The inclusion $\widehat{G}_L \subseteq U_V$ for all $V$ follows from the density of $G x$ in $\widehat{G}_L$ for all $x \in \widehat{G}_L$.
  Conversely, assume that $q = x^* y \in \bigcap_V U_V$ with $x,y \in \widehat{G}_L$.
  Then there exist $g_n \in G$ such that $g_n^{-1} x^* y = (x g_n)^* y \rightarrow 1$. By passing to a subsequence, we may assume that $xg_n \to z \in R$ and then, by continuity, $z^*y = 1$. Let $h_m \to z$ with $h_m \in G$. Then $h_m^{-1}y \to 1$, implying that $d_L(h_m^{-1}y, 1) \to 0$ and $h_m \to y$. Therefore $x g_n \rightarrow y$.
  Since $g \mapsto xg$ is isometric in $d_L$, $(g_n)_n$ is Cauchy, say $g_n \rightarrow z \in \widehat G_L$.
  We obtain that $xz = y$, so $q = x^*y = z \in \widehat{G}_L$.
  
  It follows that $\bigcap_V (U_V \cap U_V^*) = \widehat{G}_L \cap \widehat{G}_L^* = G$.
\end{proof}

\begin{lem}
  \label{l:K-dense-subgroup}
  Suppose that $G$ is a Roelcke precompact Polish group that satisfies $R(G) = W(G)$.
  Let $e$ be a central idempotent in $W(G)$ and $K = \set{p \in W(G) : pe = e}$. Then $G \cap K$ is dense in $K$.
\end{lem}
\begin{proof}
  For an open set $1 \in V \sub W$, let
  \begin{equation*}
    U_V = \{q \in W : Gq \cap V \neq \emptyset\} =
    \{q \in W : Wq \cap V \neq \emptyset\}.
  \end{equation*}
  (The equality holds because $G$ is dense in $W$.)
  By \autoref{lem:ApproximatelyInvertible} and the Baire category theorem, it suffices to show that for every open $V \ni 1$, $U_V \cap K$ and $U_V^* \cap K$ are dense in $K$.

  Let $U \subseteq W$ be open such that $K \cap U \neq \emptyset$, say $p = x^* y \in K \cap U$, with $x, y \in \widehat G_L$.
  Let $q \in K$ satisfy $qx = y$ as per \autoref{l:central-idemp}.
  Since $(qx)^* q x = y^* y \in V$ and $Wx = W$, by \autoref{th:joint-cont}, there exists an open neighbourhood $V_0 \ni x$ such that $V_0^* q^* q V_0 \subseteq V$.
  Similarly, as $x^* q x \in U$, there exists an open neighbourhood $U_0 \ni x$ such that $U_0^* q U_0 \subseteq U$.
  Let $g \in V_0 \cap U_0 \cap G$ and $q' = g^{-1} q g$.
  Then $q' \in U$, $q' \in K$ (as $e$ is central), and $q'^* q' = g^{-1} q^* q g \in V$ implies that $q' \in U_V$ (because $(g^{-1}q^*g)q' = g^{-1}q^*g g^{-1}qg = g^{-1}q^*qg \in V$), as desired.

  The argument that $U_V^* \cap K$ is dense is symmetric.
\end{proof}

The following proposition is the main result of this section and will be key for the proof of our main theorem in Section~\ref{sec:minimality-topology}.
\begin{prp}
  \label{p:open-map}
  Let $G$ be a Roelcke precompact Polish group satisfying $R(G) = W(G)$. Suppose that $e$ is a central idempotent in $W = W(G)$ and let $K = \set{p \in W : pe = e}$. Then $N = K \cap G$ is a closed normal subgroup of $G$ and the map $\pi \colon G/N \to H(e)$, $gN \mapsto ge$ is an isomorphism between $G/N$ and $H(e)$.
\end{prp}
\begin{proof}
  As $G/N$ and $H(e)$ are Polish groups and $\pi$ has a dense image, it suffices to check that $g_k e \to e$ implies that $g_kN \to N$ for any sequence $\set{g_k}_k \sub G$. By passing to a subsequence, we may assume that $g_k \to w \in K$. By Lemma~\ref{l:K-dense-subgroup}, there also exists a sequence $\set{n_k}_k \sub N$ with $n_k \to w$. By the definition of the Roelcke uniformity \eqref{eq:dRoelcke}, there exists a sequence $\set{f_k}_k \sub G$ such that $d_R(g_k, f_k n_k f_k^{-1}) \to 0$, showing that $g_kN \to N$.
\end{proof}

\section{WAP group topologies and minimality}
\label{sec:minimality-topology}

In this section, we study continuous homomorphisms defined on Roelcke precompact Polish groups. The following definition captures the notion of a quotient in the category of Polish groups.
\begin{dfn}
  Let $G$ be a Polish group. A \df{quotient} of $G$ is a Polish group $G'$ together with a continuous homomorphism $\pi \colon G \to G'$ such that $\pi(G)$ is dense in $G'$. If $\pi' \colon G \to G'$ and $\pi'' \colon G \to G''$ are quotients of $G$, say that $G''$ is a quotient of $G'$ if there exists a continuous homomorphism $\rho \colon G' \to G''$ such that $\pi'' = \rho \circ \pi$. We say that a quotient $(G', \pi)$  of $G$ is \emph{injective} or \emph{surjective} if the map $\pi$ is.
\end{dfn}
For example, if $N \unlhd G$ is a closed normal subgroup, then $G/N$ is a Polish group and $G \to G/N$ is a surjective quotient; conversely, every surjective quotient is of this type. Every quotient $\pi \colon G \to G'$ factors as a composition $G \to G/\ker \pi \to G'$, the first map being surjective and the second injective. It is an easy consequence of the definition that a quotient of a Roelcke precompact group is Roelcke precompact (see, e.g., \cite{Tsankov2012}*{Proposition~2.2}).

\begin{exm}
  Let $\Homeo^+([0, 1])$ denote the group of orientation-preserving homeomorphisms of the interval; equipped with the uniform convergence topology, it is a Polish group. If we view the interval as the set of Dedekind cuts of the rationals, we see that every order-preserving automorphism $\Aut(\Q)$ induces a homeomorphism of $[0, 1]$; this defines a homomorphism $\Aut(\Q) \to \Homeo^+([0, 1])$ with a dense image and thus we see that $\Homeo^+([0, 1])$ is an (injective) quotient of $\Aut(\Q)$.
\end{exm}
As $\Aut(\Q)$ is a Roelcke precompact Polish group (Example~\ref{ex:Q}), this example provides a negative answer to Question~4.41 of Dikranjan--Megrelishvili~\cite{Dikranjan2013}, who asked whether every Roelcke precompact Polish group is minimal.

A topological group $G$ is called \df{totally minimal} if every continuous surjective homomorphism from $G$ to a Hausdorff topological group is an open map. If a Polish group $G$ is totally minimal, then every continuous homomorphic image of $G$ in another Polish group is closed, and in fact, total minimality is equivalent to this.
\begin{prp}
  \label{p:minimal-Polish}
  Let $G$ be a Roelcke precompact Polish group. Then $G$ is totally minimal if and only if every Polish quotient of $G$ is surjective.
\end{prp}
\begin{proof}
  If $G$ is totally minimal and $\pi \colon G \to G'$ is a quotient, then $\pi$ is an open map on its image, so $\pi(G)$ is a Polish group, which, being dense in $G'$, must coincide with $G'$.

  Conversely, suppose that every Polish quotient of $G$ is surjective and let $\pi \colon G \to H$ be a continuous surjective homomorphism to the Hausdorff topological group $H$. Then $H$ is also Roelcke precompact and $\pi$ extends to a continuous surjective map $R(G) \to R(H)$. As $R(G)$ is metrizable, $R(H)$ also is and, as the inclusion $H \hookrightarrow R(H)$ is a homeomorphic embedding, $H$ is also metrizable. The group $H$, being the continuous image of the separable group $G$, is separable and therefore the two-sided completion $\widehat H$ is a Polish group. By our hypothesis, the composition $G \xrightarrow{\pi} H \to \widehat H$ is surjective, showing that $H$ is Polish and $\pi$ is open.
\end{proof}

Say that a topological group $G$ is \df{WAP faithful} if WAP functions on $G$ separate points from closed sets (or, equivalently, the compactification map $G \to W(G)$ is a homeomorphic embedding). Those groups are sometimes called \emph{reflexively representable} because they are exactly the groups that admit topologically faithful representations by isometries on a reflexive Banach space \cite{Megrelishvili2001b}. Say that a quotient $G'$ of $G$ is a \df{WAP quotient} if it is WAP faithful. The following lemma is standard.

\begin{lem}
  \label{l:Rp-quotients}
  Let $G$ be Polish Roelcke precompact group that satisfies $R(G) = W(G)$. Then every quotient $G'$ of $G$ also satisfies $R(G') = W(G')$ and is therefore a WAP quotient.
\end{lem}
\begin{proof}
  Let $\pi \colon G \to G'$ be a quotient of $G$. Then $G'$ is also Roelcke precompact and if $f$ is a uniformly continuous function on $G'$, $f \circ \pi$ is uniformly continuous on $G$, hence WAP on $G$, and therefore WAP on $G'$. As uniformly continuous functions always separate points from closed sets, we obtain the second conclusion of the lemma.
\end{proof}

Let $G$ be a topological group. Define a partial preorder on the set of quotients of $G$ as follows: $(G_1, \pi_1) \prec (G_2, \pi_2)$ if there exists a compact normal subgroup $K \unlhd G_1$ and a continuous homomorphism $\rho \colon G_2 \to G_1 / K$ such that $\rho \circ \pi_2 = \sigma \circ \pi_1$, where $\sigma \colon G_1 \to G_1 / K$ denotes the factor map. If $G_2 \prec G_1$ and $G_3 \prec G_2$ as witnessed by $\rho_1 \colon G_1 \to G_2 / K_2$ and $\rho_2 \colon G_2 \to G_3/K_3$, then $\rho_2' \circ \rho_1 \colon G_1 \to G_3 / (K_3 \rho_2(K_2))$, where $\rho_2' \colon G_2/K_2 \to G_3 / (K_3 \rho_2(K_2))$ is the factor of $\rho_2$, witnesses that $G_3 \prec G_1$. Say that two quotients $G_1$ and $G_2$ of $G$ are \emph{$\sim$-equivalent} if $G_1 \prec G_2$ and $G_2 \prec G_1$. Let $\mcQ(G)$ denote the set of $\sim$-equivalence classes of WAP quotients of $G$. Then $\prec$ naturally descends to a partial order on $\mcQ(G)$.

Let $\pi \colon G \to G'$ be a quotient of $G$.
By the universal property of $W(G)$, $\pi$ extends to a homomorphism $\overline \pi\colon W(G) \to W(G')$.
Then $\overline \pi^{-1}(1_{G'})$ is a closed subsemigroup of $W(G)$ stable under the involution and by \autoref{l:LeastIdempotent}, has a least idempotent; define $E(G')$ to be this least idempotent.
It is clear that $E(G')$ is central in $W(G)$. Observe that if $G_2$ is a quotient of $G_1$, then $E(G_2) \leq_L E(G_1)$.

Let $e$ be a central idempotent in $W = W(G)$. Recall from Section~\ref{sec:wap-compactification} that
\[
H(e) = \set{h \in W : he = eh = h \And h h^* = h^*h = e}
\]
is a Polish group with identity $e$.
Let $\overline \pi_e \colon W \to We$ denote the homomorphism $p \mapsto pe$, and let $\pi_e\colon G \rightarrow H(e)$ denote its restriction to $G$.
Since $G$ is dense in $W$, $\pi_e(G) \sub H(e)$ is dense in $We$, i.e., $We = \cl{H(e)}$.
A fortiori $\pi_e(G)$ is dense in $H(e)$, that is, $H(e)$ is a quotient of $G$.
We will show below that one obtains all WAP quotients of $G$, up to $\sim$, in that way.

The following proposition is based on ideas of Ruppert~\cite{Ruppert1990}.
\begin{prp}
  \label{p:corr-quot}
  Let $G$ be a Polish Roelcke precompact group.
  Then the map $e \mapsto H(e)$ gives an isomorphism between the set of central idempotents in $W(G)$ ordered by $\leq_L$ and $(\mcQ(G), \prec)$, with inverse $[G'] \mapsto E(G')$ (where $[G']$ denotes the class of $G'$ modulo $\sim$).
  Moreover, every $G' \sim H(e)$ is isomorphic to $H(e)/K$ for some compact normal $K \leq H(e)$.
\end{prp}
\begin{proof}
  First, let us check that the map $E$ is well-defined on $\mcQ(G)$.
  Indeed, assume that $G_1$ is a quotient of $G$ and $K_1 \unlhd G_1$ is compact, and consider the following maps:
  \[
  W(G) \xrightarrow{\ \overline \pi\ } W(G_1) \xrightarrow{\ \rho\ } W(G_1 / K_1).
  \]
  We verify that $\rho^{-1}(1_{G_1 / K_1}) = K_1$. Let $p \in W(G_1)$ be such that $\rho(p) = 1_{G_1 / K_1}$ and let $g_n \to p$, $g_n \in G_1$. Then $\rho(g_n) \to 1_{G_1 / K_1}$ and there exists a sequence $\set{k_n}_n \sub K_1$ such that $g_n k_n \to 1_{G_1}$.
  Since $K_1$ is compact, we may assume that $k_n \to k \in K_1$, so that $p = k^{-1} \in K_1$. Now suppose that $e \in (\rho \circ \overline \pi)^{-1}(1_{G_1/K_1}) = \overline \pi^{-1}(K_1)$ is an idempotent. Then $\overline \pi(e) \in K_1$ is also an idempotent, so $\overline \pi(e) = 1_{G_1}$, whence $e \in \overline \pi^{-1}(1_{G_1})$.
  Therefore $E(G_1) = E(G_1 / K_1)$.

  We next check that $H(E(G')) \sim G'$ for any WAP quotient $\pi \colon G \to G'$.
  Indeed, let $e = E(G')$, $S = \overline \pi^{-1}(1_{G'})$ and $K = S \cap We$.
  Let $k \in K$. As $k \in We$, $k \leq_L e$ and as $e$ is a $\leq_L$-least element of $S$, $k$ is $\leq_L$-least, too. In particular, $k \equiv_L ek$ and by Lemma~\ref{l:jcont-conseq}, $k = e^* e k = ek = ke$.
  Applying Lemma~\ref{l:jcont-conseq} again, from $ke \equiv_L e$, we obtain that $k^* k = k^* k e = e$. By a symmetric argument, $k k^* = e$, showing finally that $k \in H(e)$.
  We conclude that $K$ is a compact subgroup of $H(e)$. It is also normal because $S$ is invariant under conjugations by $G$ (and therefore, by $H(e)$, the set $\set{ge : g \in G}$ being dense in $H(e)$).
  Next we show that $\tilde \pi = \overline \pi |_ {H(e)}$ is a surjective homomorphism onto $G'$. As $\ker \tilde \pi = K$ and $H(e)/K$ and $G'$ are both Polish groups, it suffices to check that for any sequence $\set{h_n}_n \sub H(e)$, if $\tilde \pi(h_n) \to 1_{G'}$, then $h_nK \to K$. Let $h_n \to s \in W$. Then $s \in S \cap We = K$, and we are done.

  Next we verify that $E(H(e)) = e$ for every central idempotent $e \in W$. By the universal property of $W(G)$ and $W(H(e))$, the maps $G \xrightarrow{\pi_e} H(e) \hookrightarrow We$ give rise to maps
  \[
  W \longrightarrow W(H(e)) \longrightarrow We
  \]
  whose composition is necessarily $\overline \pi_e\colon W \rightarrow We$.
  Since $\overline \pi_e |_{We}$ is the identity, the map $W(H(e)) \to We$ is an isomorphism.
  It only remains to observe that $e$ is a $\leq_L$-least element in the semigroup $S = \set{w \in W : we = e}$.

  Finally, we note that $e \mapsto H(e)$ is order-preserving. Indeed, if $e_1 \leq_L e_2$, then the map $H(e_2) \to H(e_1)$, $h \mapsto he_1$ is a quotient map.
\end{proof}

Now we have all of the ingredients to prove one of our main theorems.
\begin{thm}
  \label{th:min-topology}
  Let $G$ be a Roelcke precompact Polish group such that $R(G) = W(G)$. Then $G$ is totally minimal.
\end{thm}
\begin{proof}
  Let $\pi \colon G \to G'$ be a quotient of $G$. By Lemma~\ref{l:Rp-quotients}, $G'$ is a WAP quotient; let $e = E(G')$. By Proposition~\ref{p:corr-quot}, $\pi$ splits as
  \[
  G \longrightarrow H(e) \longrightarrow G',
  \]
  where the second map is surjective. By Proposition~\ref{p:open-map}, the first one is also surjective, so $\pi$ is surjective, completing the proof.
\end{proof}

In certain cases, we can calculate $W(G)$ even if the group does not satisfy $W(G) = R(G)$. Then Proposition~\ref{p:corr-quot} still applies and allows us to characterize the WAP quotients of $G$, see Section~\ref{sec:exampl-wap-comp} for some examples.

\begin{prp}
  \label{p:approx-wap}
  Let $G$ be a Roelcke precompact Polish group and $H_0 \geq H_1 \geq \cdots$ be a sequence of closed subgroups of $G$ such that for every neighbourhood $U \ni 1_G$, there exists $n$ such that $H_n \sub U$. Then the set
  \[
  \set{f \in \WAP(G) : \exists n \ f \text{ is constant on double cosets of } H_n}
  \]
  is dense in $\WAP(G)$.
\end{prp}
\begin{proof}
  Let $W = W(G)$, let $f \in C(W)$ be arbitrary, and let $\eps > 0$. Consider first the action of $G$ on $W$ on the left: $g \cdot w = gw$. Let $n$ be such that for all $h \in H_n$, $\nm{h \cdot f - f} < \eps/2$. As $f$ is WAP, $\cl{H_n \cdot f}$ is weakly compact and by the Krein--\v{S}mulian theorem \cite{Conway1990}*{13.4}, $\clco(H_n \cdot f)$ is also weakly compact. Applying the Ryll-Nardzewski fixed point theorem \cite{Conway1990}*{10.8}, we obtain that there is a function $f_1 \in C(W)$ such that $\nm{f - f_1} \leq \eps/2$ and
  \begin{equation}
    \label{eq:1}
    f_1(hw) = f_1(w) \quad \text{for all } h \in H_n, w \in W.
  \end{equation}

  Next we apply the same procedure to the action on the right: $g \cdot w = wg^{-1}$. Let $H_m \leq G$ be such that $\nm{h \cdot f_1 - f_1} < \eps/2$ for all $h \in H_m$. As the left and the right action of $G$ on $W$ commute, we have that \eqref{eq:1} holds for all $\psi \in \clco(H_m \cdot f_1)$. Applying again the Ryll-Nardzewski theorem, we obtain a function $f_2 \in \clco(H_m \cdot f_1)$ fixed by $H_m$, i.e., such that $f_2(h_1wh_2) = f_2(w)$ for all $h_1 \in H_n, h_2 \in H_m, w \in W$, showing that $f_2$ is constant on double cosets of $H_{\max(m, n)}$. Moreover, $\nm{f - f_2} < \eps$.
\end{proof}

\begin{thm}
  \label{thm:ZeroDimensionalWAP}
  Let $G$ be a Roelcke precompact closed subgroup of $S_\infty$ and $W$ be its WAP compactification. Then $W$ is zero-dimensional.
\end{thm}
\begin{proof}
  Since $W$ is compact, it suffices to prove that it is totally disconnected. Let $w_1, w_2 \in W$ be two distinct points and let $f \in C(W(G))$ be a function such that $f(w_1) = 0, f(w_2) = 1$. By Proposition~\ref{p:approx-wap}, there exists an open subgroup $V \leq G$ and a function $f' \in C(W)$ constant on double cosets of $V$ such that $\nm{f' - f} < 1/2$. Then $f'(w_1) \neq f'(w_2)$ and as $G$ is Roelcke precompact, $f'$ takes only finitely many values on $G$ (and therefore on $W$), showing that $f'^{-1}(\set{f'(w_1)})$ is a clopen set separating $w_1$ from $w_2$.
\end{proof}

\begin{cor}
  \label{c:wap-quot-sinfty}
  Let $G$ be a Roelcke precompact closed subgroup of $S_\infty$. Then every WAP quotient of $G$ is also isomorphic to a closed subgroup of $S_\infty$.
\end{cor}
\begin{proof}
  By Proposition~\ref{p:corr-quot}, it suffices to check that $H(e)$ is isomorphic to a subgroup of $S_\infty$ for every central idempotent $e \in W(G)$. Let $X$ be the closure of $H(e)$ in $W$. Then $H(e)$ acts on $X$ by homeomorphisms: $h \cdot x = hx$. Let $F$ be a closed subset of $H(e)$ that does not contain $e$ and let $\cl{F}$ be the closure of $F$ in $X$. Then $e \notin \cl{F}$ and as $W$ (and therefore $X$) is zero-dimensional, there exists a clopen set $U \sub X$ such that $e \in U$ and $U \cap \cl{F} = \emptyset$. Then $V = \set{h \in H(e) : hU = U}$ is a clopen subgroup of $H(e)$ separating $e$ from $F$.
\end{proof}
The last corollary can be rephrased as follows: if $G$ is a Roelcke precompact subgroup of $S_\infty$ and $\pi \colon G \to \Iso(E)$ is a continuous representation of $G$ by isometries on a reflexive Banach space $E$, then there exists a closed subgroup $H$ of $S_\infty$, a continuous homomorphism $\sigma \colon G \to H$ and a topological embedding $\iota \colon H \to \Iso(E)$ such that $\pi = \iota \circ \sigma$.

A Polish group $G$ is called \emph{WAP trivial} if it admits no non-trivial WAP quotients. The first example was found by Megrelishvili~\cite{Megrelishvili2001} who showed that $\Homeo^+([0, 1])$ is WAP trivial. Another example, due to Pestov~\cite{Pestov2007}, is $\Iso(\bU_1)$. His proof uses the result of Megrelishvili and the result of Uspenskij~\cite{Uspenskii1990} that $\Iso(\bU_1)$ is a universal Polish group; we provide a direct proof of the WAP triviality of $\Iso(\bU_1)$ in Section~\ref{sec:exampl-wap-comp}. Corollary~\ref{c:wap-quot-sinfty} gives yet another method to produce examples of WAP trivial groups.
\begin{cor}
  \label{c:wap-trivial}
  Let $H$ be a Roelcke precompact subgroup of $S_\infty$ and let $G$ be a quotient of $H$ that has no proper open subgroups. Then $G$ is WAP trivial.
\end{cor}
\begin{proof}
  Suppose that $G'$ is a WAP quotient of $G$. Then $G'$ is also a WAP quotient of $H$ and by Corollary~\ref{c:wap-quot-sinfty}, $G'$ is isomorphic to a subgroup of $S_\infty$. This gives a continuous action of $G$ on a countable set which, by hypothesis, has to be trivial.
\end{proof}
Corollary~\ref{c:wap-trivial} applies to $\Homeo^+([0, 1])$ (with $H = \Aut(\Q)$), thus providing a new proof of Megrelishvili's result, but also to some other homeomorphism groups.

For example, using the method of projective Fraïssé limits (see \cite{Irwin2006}), Barto\v{s}ova and Kwiatkowska~\cite{Bartosova2013p} construct a homomorphism from a Roelcke precompact closed subgroup of $S_\infty$ to the homeomorphism group $\Homeo(\bL)$ of the \df{Lelek fan} with a dense image and also show that $\Homeo(\bL)$ has no proper open subgroups; thus, we conclude that $\Homeo(\bL)$ is WAP trivial.

In \cite{Glasner2008}*{Question~10.5}, Glasner and Megrelishvili ask for the existence of a group which is WAP trivial but does not contain a copy of $\Homeo^+([0, 1])$. In fact, $\Homeo(\bL)$ is such a group: by \cite{Bartosova2013p}, $\Homeo(\bL)$ is totally disconnected (and therefore does not contain a copy of $\Homeo^+([0, 1])$) and by the above remark, it is WAP trivial. We are grateful to Michael Megrelishvili for pointing that out.

We finally give a proof of Corollary~\ref{c:I:unique-topology} from the introduction. Recall that a Polish group $G$ has \emph{ample generics} if the conjugation action $G  \actson G^n$ has a comeager orbit for every $n$. (Note that the definitions of ample generics given in \cite{Hodges1993a} and \cite{Kechris2007a} are somewhat different; we use the one of \cite{Kechris2007a}.)
\begin{lem}
  \label{l:ample-generics}
  Let $G$ be the automorphism group of a classical, $\aleph_0$-categorical, $\aleph_0$-stable structure. Then $G$ admits a basis of neighbourhoods of the identity consisting of open subgroups with ample generics.
\end{lem}
\begin{proof}
  Let $M$ be such a structure. $M\eq$ can be equipped with a relational language in which it eliminates quantifiers, and can then be transformed into a $1$-sorted homogeneous structure by naming each sort with a unary predicate. Let $\mcK$ be the age of this structure, and let $\mcK_p^n$ be defined as in the paragraph preceding Theorem~2.11 in \cite{Kechris2007a}. First apply \cite{Hodges1993a}*{Proposition~3.4} to see that $M$ has an \emph{amalgamation base} $\mcA$ \cite{Hodges1993a}*{Definition~2.8}.
  If $A \in \mcA$, then $A$ is of the form $\acl\eq(C)$ for some finite $C \sub M$.
  By \cite{Hodges1993a}*{Theorem~3.1}, $A$ is then interdefinable with a finite subset of itself, and in particular, $A$ is interdefinable with its restriction to a finite family of sorts, which is both finite and closed under all automorphisms of $A$, so in what follows we may assume that $\mcA \subseteq \mcK$.
  From the existence of an amalgamation base, it follows that $\mcK_p^n$ has the \emph{weak amalgamation property} \cite{Kechris2007a}*{Definition~3.3} and the \emph{cofinal joint embedding property (CJEP)}  \cite{Kechris2007a}*{Definition~2.13}.
  Moreover, CJEP holds uniformly in $n$, in the sense that for every $A \in \mcK$ there is $B \geq A$, $B \in \mcK$ (in fact, $B \in \mcA$) such that for every $n$, $\langle B, \id_B, \ldots, \id_B \rangle$ is a witness for the corresponding instance of CJEP in $\mcK_p^n$.
  By \cite{Kechris2007a}*{Theorem~3.9}, which also applies for $n$-tuples of automorphisms, WAP and CJEP for $\mcK_p^n$ imply that $G$ admits a basis of neighbourhoods of $1_G$ consisting of subgroups $H$ such that $H^n$ has a comeagre orbit under conjugation.
  By the uniform CJEP, the same argument yields that $G$ admits a basis of neighbourhoods of $1_G$ consisting of subgroups $H$ such that for all $n$, $H^n$ has a comeagre orbit under conjugation, and we are done.
\end{proof}

On a related note, Malicki~\cite{Malicki2015} recently characterized the ultrametric spaces whose isometry groups satisfy the conclusion of Lemma~\ref{l:ample-generics}, that is, admit a basis at the identity consisting of open subgroups with ample generics.

\begin{proof}[Proof of Corollary~\ref{c:I:unique-topology}]
  Let $G$ be the automorphism group of an $\aleph_0$-categorical, $\aleph_0$-stable structure and denote by $\tau$ its standard Polish topology. Suppose that $\sigma$ is some other separable Hausdorff topology on $G$. Let $H$ be an open subgroup of $G$ with ample generics as given by Lemma~\ref{l:ample-generics}. Then, by \cite{Kechris2007a}*{Theorem~6.24}, $\id \colon (H, \tau) \to (H, \sigma)$ is continuous, so $\id \colon (G, \tau) \to (G, \sigma)$ is continuous. By Theorem~\ref{th:min-topology}, it is also open, so a homeomorphism.
\end{proof}

\section{The model-theoretic viewpoint}
\label{sec:model-theor-viewp}

In this section we describe the model-theoretic meaning of many notions and results appearing in earlier sections --- and in fact, several key results of this paper were first given model-theoretic proofs that were only later translated into the language of semigroups.
We shall also assume some familiarity with model theory, including metric model theory.
In the context of the latter, we ignore the distinction between a formula and a definable predicate (uniform limit of formulas), which is purely syntactic.
Throughout, let $\bM$ be an $\aleph_0$-categorical structure and let $G = \Aut(\bM)$.

The Ryll-Nardzewski Theorem allows us to recover definable predicates and type spaces over $\emptyset$ purely from the group action $G \actson M$.
Indeed, an $\aleph_0$-categorical structure is approximately homogeneous, which means that two tuples $a,b \in M^n$ have the same type if and only if $[a] = [b]$ in $M^n \dslash G$.
In addition, all $n$-types are realised in $\bM$, so $M^n \dslash G = \tS_n(\emptyset)$, the space of $n$-types over $\emptyset$, and similarly, for an arbitrary countable index set $I$ instead of $n$.
By the Ryll-Nardzewski Theorem (or by \autoref{thm:RoelckePreCompactGroup}), the logic topology on $\tS_I(\emptyset) = M^I \dslash G$ agrees with the topology induced by the metric, and it is compact.
Therefore, an $I$-ary formula is just a continuous function on $M^I \dslash G$, or equivalently, a continuous $G$-invariant function $\varphi \colon M^I \rightarrow \bC$, and every such function is automatically uniformly continuous and bounded (for $I$ countable --- and a formula cannot depend on more than countably many variables).
A continuous combination of formulas is, of course, a formula; as for quantification, this is entirely subsumed in \autoref{thm:RoelckePreCompactGroup}.
Similarly, if $X \subseteq M^I$ is $G$-invariant and closed, we can speak of a \emph{formula on $X$} as being either the restriction of an $I$-ary formula to $X$ or a continuous function on $X \dslash G$; by the Tietze Extension Theorem, the two notions agree.
If $\bM$ is a classical structure, then $M^I \dslash G$ is totally disconnected and $\{0,1\}$-valued formulas suffice to describe the logic.

In order to define stability, one considers formulas in two groups of variables, i.e., formulas on $X \times Y$ where $X, Y \sub M^I$ are $G$-invariant and closed. Recall that a real-valued formula $\varphi$ has the \df{order property} on $X \times Y$ if there exist sequences $\set{x_n}_n \sub X$ and $\set{y_m}_m \sub Y$ and real numbers $r < s$ such that $\varphi(x_n, y_m) \leq r$ for $n < m$ and $\varphi(x_n, y_m) \geq s$ for $n > m$, or the other way round. A formula is \df{stable} if it does not have the order property; a theory is \df{stable} if every formula is stable on $\bM^I \times \bM^I$. By passing to appropriate subsequences, it is easy to check that the absence of the order property is equivalent to Grothendieck's condition:
\begin{equation} \label{eq:stability}
\lim_m \lim_n \varphi(x_m, y_n) = \lim_n \lim_m \varphi(x_m, y_n)
\end{equation}
for all sequences $\set{x_n}_n \sub X, \set{y_m}_m \sub Y$ for which both limits exist. For complex-valued formulas, it will be convenient to take \eqref{eq:stability} as the definition of stability. Now we have the following.
\begin{lem}
  \label{lem:StableWAP}
  Let $\varphi$ be a formula on $X \times Y$.
  Then $\varphi$ is stable if and only if the function $\tilde \varphi_{x,y}\colon G \to \bC$ defined by $\tilde \varphi_{x, y}(g) = \varphi(x,g \cdot y)$ is WAP for every $x \in X$ and $y \in Y$.
  If $X$ and $Y$ are both orbit closures, say $X = [x_0]$ and $Y = [y_0]$, then $\varphi$ is stable if and only if $\tilde \varphi_{x_0,y_0}$ is WAP.
\end{lem}
\begin{proof}
  As $\tilde \varphi_{x_0, y_0}(g^{-1} h) = \varphi(x_0,g^{-1} h \cdot y) = \varphi(g \cdot x_0,h \cdot y_0)$, the second assertion follows directly from \autoref{th:GrothendieckGroups}.

  To prove the first assertion, observe that if $\tilde \varphi_{x, y}$ violates \eqref{eq:GrothendieckGroups} for some sequences $\set{g_m}, \set{h_n} \sub G$, then the sequences $g_m^{-1} \cdot x$ and $h_n \cdot y$ violate \eqref{eq:stability}. Conversely, suppose that $\set{x_m}$ and $\set{y_n}$ are sequences that violate \eqref{eq:stability}. Using the compactness of $X \dslash G$ and $Y \dslash G$, we may assume that $[x_m] \to [x_0]$ and $[y_n] \to [y_0]$, i.e., there exist sequences $\set{g_m}$ and $\set{h_n} \sub G$ such that $d(g_m \cdot x_0, x_m) < 2^{-m}$ and $d(h_n \cdot y_0, y_n) < 2^{-n}$. Now using the uniform continuity of $\varphi$ and by passing to subsequences again, we obtain that $\tilde \varphi_{x_0, y_0}$ violates \eqref{eq:GrothendieckGroups}.
\end{proof}

One can develop a large part of stability theory in this formalism --- here we shall content ourselves with pointing out how the definability of types follows from Grothendieck's criterion.
For this, we shall require a slightly stronger form of \autoref{th:GrothendieckGroups}.
\begin{thm}[\cite{Grothendieck1952}*{Th{\'e}or{\`e}me~6}]
  \label{thm:GrothendieckCriterion}
  Let $X$ be any topological space, let $A \subseteq C_b(X)$, and let $X_0 \subseteq X$ be a dense set.
  Then $A$ is weakly precompact in $C_b(X)$ if and only if $A$ is uniformly bounded and for every two sequences $\{f_n\} \subseteq A$ and $\{x_m\} \subseteq X_0$, we have
  \begin{gather*}
    \lim_n \lim_m f_n(x_m) = \lim_m \lim_n f_n(x_m)
  \end{gather*}
  as soon as both limits exist.
\end{thm}

Fix a formula $\varphi$ on $X \times Y$ as above.
For $x \in X$, $y \in Y$, define $\varphi_x\colon Y \rightarrow \bC$ and $\varphi^y\colon X \rightarrow \bC$ by
\begin{gather*}
  \varphi_x(y) = \varphi^y(x) = \varphi(x,y).
\end{gather*}
Define $\mcD_\varphi \subseteq C_b(X)$ (respectively, $\mcD^\varphi \subseteq C_b(Y)$) as the C$^*$-algebra generated by $\{ \varphi^y : y \in Y\}$ (respectively, $\{ \varphi_x : x \in X \}$), and let $\tS_\varphi$ (respectively, $\tS^\varphi$) denote its Gelfand space. Observe that as $\varphi$ is uniformly continuous, both maps $X \to \mcD^\varphi$, $x \mapsto \varphi_x$ and $Y \to \mcD_\varphi$, $y \mapsto \varphi^y$ are continuous. For each $p \in \tS_\varphi$, define the function $\varphi_p \in C_b(Y)$ by $\varphi_p(y) = \varphi^y(p)$. Thus $\tS_\varphi$ is the space of $\varphi$-types in $X$ over $Y$, while $\tS^\varphi$ is the space of $\varphi'$-types in $Y$ over $X$, where $\varphi'(y,x) = \varphi(x,y)$ is the transposed formula. As $\mcD_\varphi$ consists of continuous functions on $X$, we can consider $S_\varphi$ as a (not necessarily faithful) compactification of $X$ with a natural map $\theta \colon X \to \tS_\varphi$ with a dense image ($\theta(x)$ is just the $\varphi$-type of $x$ over $Y$ and realised types are dense). As an illustration of how one can work in this framework, we give (the almost tautological) proof of one of the basic results in stability theory: that for a stable formula $\varphi$, $\varphi$-types are definable.
\begin{prp}
  Assume $\varphi$ is stable.
  Then for every $p \in \tS_\varphi$, the function $\varphi_p$ agrees (on $Y$) with some $\psi \in \mcD^\varphi$, i.e., with some continuous combination of instances $\varphi_{x_i}$.
\end{prp}
\begin{proof}
  Let $p \in \tS_\varphi$, and let $x_k \in X$ be such that $\theta(x_k) \rightarrow p$ in $\tS_\varphi$.
  Let $A = \{\varphi_x : x \in X\} \subseteq C(\tS^\varphi) \subseteq C_b(Y)$.
  Using the fact that $\varphi$ is stable and \autoref{thm:GrothendieckCriterion}, we obtain that $A$ is weakly precompact in $C(\tS^\varphi)$.
  Using the Eberlein--\v{S}mulian Theorem, and possibly passing to a subsequence, we may assume that $\varphi_{x_k} \rightarrow \psi$ weakly for some $\psi \in C(\tS^\varphi) = \mcD^\varphi$. As weak convergence in $C(\tS^\varphi)$ implies pointwise convergence, we have that
\[
\varphi_p(y) = \varphi^y(p) = \lim_k \varphi^y(x_k) = \lim_k \varphi_{x_k}(y) = \psi(y)
\]
for all $y \in Y$.

  (When $X$ and $Y$ are both orbit closures (i.e., complete pure types), the action of $G$ is topologically transitive on both and \autoref{th:GrothendieckGroups} suffices for this argument.)
\end{proof}

In particular, $\varphi_p$ extends by continuity to a unique function which we can still denote by $\varphi_p \in C(\tS^\varphi)$, and by a symmetric argument we construct $\varphi^q \in C(\tS_\varphi)$ for $q \in \tS^\varphi$.
The limit exchange property tells us that $\varphi_p(q) = \varphi^q(p)$, which can be interpreted as the symmetry of independence.
This gives rise to a function $\varphi\colon \tS_\varphi \times \tS^\varphi \rightarrow \bC$ which is separately continuous but usually not jointly.

From here on, we shall fix $\xi \in M^\bN$ enumerating a dense subset and we identify $\widehat G_L$ with $\Xi = [\xi]$ as per \autoref{lem:DenseTuple}, so $\Xi$ acts on $M$.
Keeping in mind that formulas on $\Xi^2$ are the same as continuous functions on $\Xi^2 \dslash G = R(G)$, we obtain a means to calculate $W(G)$.
\begin{thm}
  \label{thm:WG}
  Let $\bM$ be $\aleph_0$-categorical and let $G = \Aut(\bM)$.
  Viewing $C\bigl( R(G) \bigr)$ as the algebra of formulas on $\Xi^2$, the subalgebra $C\bigl( W(G) \bigr)$ corresponding to the quotient map $R(G) \rightarrow W(G)$ consists exactly of the stable formulas.

  If $\bM$ is classical, let $B\bigl( R(G) \bigr) \subseteq C\bigl( R(G) \bigr)$ be the collection of continuous $\{0,1\}$-valued functions.
  Then $B\bigl( R(G) \bigr)$ (equipped with the proper operations) is the Boolean algebra of classical formulas on $\Xi^2$, it generates $C\bigl( R(G) \bigr)$, and $B\bigl( W(G) \big) = B\bigl( R(G) \bigr) \cap C\bigl( W(G) \bigr)$, the Boolean algebra of classical stable formulas, generates $C\bigl( W(G) \bigr)$.
\end{thm}
\begin{proof}
  We have already seen in Lemma~\ref{lem:StableWAP} that the WAP functions are the stable formulas.
  When $\bM$ is classical, $R(G)$, being a space of types, is zero-dimensional, and its quotient $W(G)$ is zero-dimensional by \autoref{thm:ZeroDimensionalWAP}. On zero-dimensional compact spaces, $\set{0, 1}$-valued functions separate points and therefore, by the Stone--Weierstrass theorem, the algebra generated by them is dense in the algebra of all continuous functions.
\end{proof}

In a sense, formulas on $\Xi^2$ capture the entire logic, and in particular, suffice for testing for stability.
Indeed, if $\varphi$ is a formula on $X \times Y$ and $x \in X$, $y \in Y$, then $\overline \varphi_{x,y}(x',y') = \varphi(x' \cdot x, y' \cdot y)$ is a formula on $\Xi^2$, and $\varphi$ is stable if and only if $\overline \varphi_{x,y}$ is stable for all such $x,y$. Hence, we obtain the following.
\begin{thm}
  \label{th:stable-equiv-r-equal-w}
  Let $\bM$ be an $\aleph_0$-categorical structure and $G = \Aut(\bM)$. The following are equivalent:
  \begin{enumerate}
  \item the theory of $\bM$ is stable;
  \item every formula on $\Xi^2$ is stable;
  \item $R(G) = W(G)$.
  \end{enumerate}
\end{thm}

Stability-theoretic independence admits many equivalent characterisations.
For our purposes, let us consider $x,y,z \in \Xi$, and for a stable formula $\varphi$, let $p_\varphi = \tp_\varphi(x/y) \in \tS_\varphi$, i.e.,
\[
\varphi(p_\varphi,u) = \varphi(x,yu) = \tilde \varphi( x^* yu ) \quad \text{for } u \in \Xi,
\]
and similarly, we let $q^\varphi = \tp_{\varphi'}(z/y) \in \tS^\varphi$, i.e., $\varphi(u,q^\varphi) = \varphi(yu, z)$.
We then say that $x \ind_y z$ if $\varphi(x,z) = \varphi(p_\varphi,q^\varphi)$ for every formula $\varphi$.

\begin{rmk*}
  Usual stability independence is stronger than what is defined above, and requires that $\tp_\varphi(x/y,z)$ be definable over $y$.
  In our vocabulary, this means that we consider stable formulas $\varphi$ on $\Xi \times \Xi^2$, express $\varphi_{p_\varphi}$ as $\psi_\xi$, and require that $\varphi(x,y,z) = \psi(y,y,z)$ for all such $\varphi$ (when $\varphi(x,y,z)$ only depends on $x,z$ we have $\psi(y,y,z) = \varphi(p_\varphi,q^\varphi)$, this is implicit in the proof of \autoref{prp:IndependenceLaw}).
  However, first, this weaker notion suffices to characterise the semi-group law in $W(G)$, and second, the triplet $x,y,z$ constructed in the proof below actually does satisfy the stronger form, so there is no real cheating here.
  Moreover, the two notions agree when $y$ can be expressed as either $x \cdot w$ or $z \cdot w$, which is the only case considered explicitly in the remarks following \autoref{prp:IndependenceLaw}.)
\end{rmk*}

\begin{prp}
  \label{prp:IndependenceLaw}
  Let $p,q \in R(G)$.
  Then
  \begin{enumerate}
  \item There are $x,y,z \in \Xi$ such that $p = x^* y$, $q = y^* z$, and $x \ind_y z$.
  \item Whenever $x \ind_y z$ we have $x^* z = (x^* y) \cdot (y^* z)$ in $W(G)$.
  \end{enumerate}
  Thus $\tp(x,y) \cdot \tp(y,z) = \tp(x,z)$ when $x \ind_y z$, and the semigroup $W(G)$ is an algebraic representation of stable independence in $\bM$.
\end{prp}
\begin{proof}
  For any $x,y,z$ as in the first item, the image of $x^* z$ in $W(G)$ depends only on $p,q$.
  It will therefore be enough to show there exist $x,y,z$ satisfying both items.
  For this, let $p = x_0^* y_0$ and $\{g_k\} \subseteq G$, $g_k \rightarrow q$.
  Possibly passing to a sub-sequence, we may assume that $[x_0,y_0,y_0 g_k]$ converges in $\Xi^3 \dslash G$ to some $[x,y,z]$ (we say that $\tp(z/x,y)$ is a \emph{co-heir} over $y$).
  Then $x^* y = p$, $y^* z = \lim y_0^* y_0 g_k = q$ and $p \cdot q = \lim x^* y g_k = \lim x_0^* y_0 g_k = x^* z$.
  Notice that for stable $\varphi$ and $u \in \Xi$ we have $\varphi(u,q^\varphi) = \varphi(yu,z) = \lim \varphi(u,g_k)$.
  As we saw, we can express $\varphi_{p_\varphi}$ as $\lambda\bigl[ \varphi_{u_i} \bigr]_{i \in \bN}$, where $\lambda\colon \bC^\bN \rightarrow \bC$ is continuous, which means that
  \begin{gather*}
    \varphi(p_\varphi,q^\varphi) = \lim_k \varphi(p_\varphi,g_k) = \lim \varphi(x,yg_k) = \varphi(x,z),
  \end{gather*}
  and we are done.
\end{proof}

While the presentation is entirely symmetric, it becomes convenient at some point to ``fix sides'' and think of $p = x^*y$ as $\tp(y/x)$ (rather than, say, $\tp(x/y)$).
Of course, $p$ does not determine $x$, so we may view $p$ as a type over any $u \in \Xi$ (which has the same type as $x$), i.e., over any elementary substructure $u(\bM) \preceq \bM$.
Then $u^* p = u^* x^* y$ corresponds to $\tp(y/xu)$, i.e., taking the product $u^* p$ (which also makes perfect sense in $R(G)$) corresponds to the restriction of the parameter set from $x$ to $xu$, or from $M$ to $u(M)$.
On the other hand, $up = u x^* y$ only makes sense in $W(G)$, and corresponds to the unique non-forking extension of a type from $xu$ to $x$, or from $u(M)$ to $M$.
A general product $(u^*v)(x^* y)$ in $W(G)$ consists therefore of taking a non-forking extension followed by a restriction.

We now turn to French school stability and the fundamental order.
We can define the \emph{fundamental class} of $p = x^* y \in W(G)$, denoted by $\beta(p)$, as the set of all stable formulas $\varphi$ which are ``almost represented'' in $\tp(y/x)$, i.e.,
\begin{gather*}
  \beta(p) = \big\{ \varphi \text{ stable} : \inf_{u \in \Xi} |\tilde \varphi( u^* p )| = 0 \big\}.
\end{gather*}
It is clear that $\beta(u^* p) \subseteq \beta(p)$ and standard considerations regarding heirs yield that $\beta(up) = \beta(p)$ for $u \in \Xi$.
Therefore, in particular, $\beta(qp) \subseteq \beta(p)$ for $p,q \in W(G)$.
Conversely, assume that $\beta(p) \supseteq \beta(q)$ and think of $p$ and $q$ as types over models $\bM_0$ and $\bM_1$, respectively.
By standard compactness arguments, one can embed both models into a bigger $\bM_2$ (e.g., an ultrapower of $\bM_0$) in such a manner that $p|^{M_2}$, the non-forking extension, also extends $q$.
In other words, $q$ is a restriction of a non-forking extension of $p$, and by our analysis of these operations, $q \in W p$.
Putting both together, we have $\beta(q) \subseteq \beta(p)$ if and only if $q \in Wp$ if and only if $Wq \subseteq Wp$, so one may identify $Wp$ with the fundamental class of $p$, and $\leq_L$ with the fundamental order of the theory of $\bM$.
With these identifications, for example \autoref{l:jcont-conseq}~\autoref{i:l:jc:1} is the fact that the fundamental order captures non-forking: a type and its extension have the same fundamental class if and only if the latter is the (unique) non-forking extension of the former.

We turn to idempotents in $W(G)$.
First, let $A \subseteq M\eq$ be some algebraically closed set, and consider two copies of $M$ which are independent over a common copy of $A$, i.e., $x,y \in \Xi$ such that $x(a) = y(a)$ for every $a \in A$ and $x \ind_{x(A)} y$.
Then, by our characterisation of the semigroup law in $W(G)$, $e_A = x^* y$ is idempotent.
Conversely, if $e = x^* y$ is idempotent, then $\mathrm{Cb}(y/x)$ can be shown to belong to $y(M\eq)$, and letting $A = x(M\eq) \cap y(M\eq)$ it follows that $e = e_A$.

An idempotent $e_A$ is central if and only if $A$ is invariant under automorphisms, but since $A$ may contain elements of various sorts, we cannot just deduce that it is $\emptyset$-definable.
However, we can relax the requirement $A = \acl\eq(A)$ to $\dcl\eq(A) = \acl\eq (A)$, in which case $A$ may even be taken to be a single countable tuple $a$.
Now, $e$ is central if and only if any $b \in [a] = \overline{G \cdot a}$ is interdefinable with $a$, so we can replace $A$ with the set $[a]$ which is definable.
For $p = x^* y$, we have $pe = e$ if and only if $ex = xe = ye = ey$, i.e., if and only if $x \equiv_A y$, and similarly, $g e = e$ if and only if $g$ fixes $A$ pointwise.
With this in mind, \autoref{l:K-dense-subgroup} says that if $x \equiv_A y$ then there is $g \in \Aut(\bM/A)$ such that $\tp(\xi,g \xi)$ is arbitrarily close to $\tp(x,y)$.
This is easily reduced to the following general model theoretic fact.
\begin{prp}
  Let $\bM$ be a stable structure (not necessarily $\aleph_0$-categorical), and let $A \subseteq \bM$ be $\emptyset$-definable.
  Let also $a,b \in \bM$ have the same type over $A$.
  Then there exists an elementary extension $\bM' \succeq \bM$ and $g \in \Aut(\bM')$ such that $ga = b$ and $g$ fixes pointwise $A$ \emph{as interpreted in $\bM'$}.
\end{prp}
The model-theoretic proof is a fairly standard elementary chain argument with the extra twist that $A$, being a definable set, grows at each induction step.
As usual for elementary chain arguments, in the topological realm this becomes a Baire category argument with the induction step more or less subsumed in \autoref{l:central-idemp}.

\section{Examples of WAP compactifications}
\label{sec:exampl-wap-comp}

If $G$ is a Roelcke precompact Polish group, one way to see whether $R(G) = W(G)$ is to check if the group operation on $G$ extends to a semigroup law on $R(G)$ (if this happens, the extension is unique). Of the examples we have considered so far, this is the case for $S_\infty$ (Example~\ref{ex:Sinfty}; the semigroup law is given by composition), $U(\mcH)$ (Example~\ref{ex:UH}; semigroup law again given by composition of operators), $\Aut(X, \mu)$ (Example~\ref{ex:AutMu}; semigroup law given by
\[
\nu_1 \nu_2 = \int_X (\nu_1)_x \times (\nu_2)^x \ud \mu(x),
\]
where $\nu_1, \nu_2$ are measures on $X \times X$ with marginals equal to $\mu$ and
\[
\nu_1 = \int_X \delta_x \times (\nu_1)_x \ud \mu(x) \quad \And \quad \nu_2 = \int_X  (\nu_2)^x \times \delta_x \ud \mu(x)
\]
are the corresponding decompositions). This reflects the fact that the theories of the corresponding structures are stable. In this section, we calculate the WAP compactification of some groups $G$ for which $W(G) \neq R(G)$. This gives information about the WAP quotients of those groups.

Calculating $W(\Aut(\bM))$ for the automorphism group of an $\aleph_0$-categorical structure $\bM$ amounts to understanding all stable formulas in $\bM$ and by virtue of Theorem~\ref{thm:WG}, if $\bM$ is a classical structure, we only need to consider classical formulas.

\begin{exm}
  \label{exm:StablePartRandomGraph}
  \textit{The random graph.} Let $(V,E)$ be the \emph{random graph} (the unique homogeneous, universal, countable graph), $E$ being the edge relation, and let $T = \Th(V,E)$.
  Let $p(\bar x), q(\bar y) \in \tS_n(T)$ be complete $n$-types, and $\varphi(\bar x,\bar y)$ a formula with two groups of $n$ variables.
  Then the following are equivalent:
  \begin{enumerate}
  \item \label{i:RG:1} There exists a quantifier-free formula $\psi(\bar x,\bar y)$ in which the symbol $E$ does not appear that agrees with $\varphi$ modulo $p(\bar x) \wedge p(\bar y)$;
  \item \label{i:RG:2} The formula $\varphi(\bar x, \bar y)$ is stable on $p(\bar x) \wedge p(\bar y)$.
  \end{enumerate}

  Let $G = \Aut(V,E)$. Then $G$ is Roelcke precompact (the action on $V$ is oligomorphic), and $W(G)$ consists of all isomorphisms between subgraphs of $(V, E)$.
\end{exm}
\begin{proof}
  Since equality is always stable, \ref{i:RG:1} $\Longrightarrow$ \ref{i:RG:2} holds.
  For \ref{i:RG:2} $\Longrightarrow$ \ref{i:RG:1}, it will be enough to show that on $p(\bar x) \wedge q(\bar y)$, if $\varphi$ is not entirely determined by the equality relations between the $x_i$ and $y_j$, then it is unstable on $p(\bar x) \wedge q(\bar y)$.
  For this, we may assume that $p(\bar x)$ requires all $x_i$ to be distinct and similarly for $q(\bar y)$.

  Say that a pair of tuples $\bar c,\bar d$ is obtained from another pair $\bar a,\bar b$ by a \emph{simple modification} if, up to isomorphism, one can be obtained from the other by adding or removing a single edge between some $(a_i,b_j)$.
  In this case, if in addition $\bar a,\bar b$ realise $p(\bar x) \wedge q(\bar y)$, then so do $\bar c,\bar d$, and conversely, given any two realisations of $p(\bar x) \wedge q(\bar y)$ such that $a_i = b_j \Longleftrightarrow c_i = d_j$, one can reach one from the other by a finite sequence of simple modifications.
  Therefore, all we need to show is that $\varphi$ is unstable as soon as there exist two realisations $\bar a,\bar b$ and $\bar c,\bar d$ of $p(\bar x) \wedge q(\bar y)$ which only differ by a simple modification, such that $\varphi(\bar a,\bar b) \wedge \neg \varphi(\bar c,\bar d)$ holds.
  Possibly replacing $\varphi$ with $\neg \varphi$, we may assume the only difference is in that $E(a_0,b_0) \wedge \neg E(c_0,d_0)$, noticing that this means that $a_0 \neq b_i$ and $b_0 \neq a_i$ for all $i$, and similarly for $\bar c,\bar d$.

  We can therefore construct a sequence $(\bar a^k,\bar b^k)_{k \in \bN}$ where $a_{>0}^k = a_{>0}$, $b_{>0}^k = b_{>0}$ for all $k$, $a_0^k$ and $b_0^k$ are all distinct from one another and from $a_{>0}$, $b_{>0}$, the quantifier-free type of $a_0^k,a_{>0},b_{>0}$ is the same as that of $a_0,a_{>0},b_{>0}$ (and therefore as $c_0,c_{>0},d_{>0}$), similarly with $b_0^k$ and $b_0$ ($d_0$) instead of $a_0^k$ and $a_0$ ($c_0$), and finally, $E(a_0^k,b_0^\ell)$ holds if and only if $k < \ell$.
  Since $T$ has quantifier elimination, we have $\varphi(\bar a^k,\bar b^\ell)$ if and only if $k < \ell$, so $\varphi$ has the order property on $p(\bar x) \wedge q(\bar y)$.

  The assertion regarding $G$ follows from the fact that $T$ is $\aleph_0$-categorical, the characterisation of stable formulas, and Theorem~\ref{thm:WG}.
\end{proof}

The next example is a continuation of Example~\ref{ex:Q}.
\begin{exm}
  \label{exm:StablePartDLOWE}
  Let $T = \Th(\Q,<)$. Let $p(\bar x), q(\bar y) \in \tS_n(T)$ be complete $n$-types, and $\varphi(\bar x,\bar y)$ a formula with two groups of $n$ variables.
  Then the following are equivalent:
  \begin{enumerate}
  \item There exists a quantifier-free formula $\psi(\bar x,\bar y)$ in which the symbol $<$ does not appear that agrees with $\varphi$ modulo $p(\bar x) \wedge q(\bar y)$;
  \item The formula $\varphi(\bar x)$ is stable on $p(\bar x) \wedge q(\bar y)$.
  \end{enumerate}

  $W(\Aut(\Q))$ consists of all isomorphisms between substructures of $\Q$.
\end{exm}
\begin{proof}
  We follow the same argument as above.
  Here, a simple modification of $\bar a,\bar b$ would consist of choosing some $a_i$ and $b_j$ which are adjacent in the order on $\bar a,\bar b$ such that $a_i \neq b_k$ and $a_k \neq b_j$ for all $k$, and inverting the relative order between them.
  The rest is identical.
\end{proof}

\begin{exm}
  \label{exm:StablePartAtomlessBooleanAlgebra}
  \textit{The Cantor space.} Let $\mcB = (B,0,1,\wedge,\vee,\neg)$ be the countable atomless Boolean algebra, and let $T = \Th(\mcB)$.
  Let $p(\bar x) \in \tS_m(T)$ and $q(\bar y) \in \tS_n(T)$ be complete types, and $\varphi(\bar x,\bar y)$ a formula with two groups of $m$ and $n$ variables, respectively.
  Then the following are equivalent:
  \begin{enumerate}
  \item There exists a formula $\psi(\bar x,\bar y)$ which is a Boolean combination of formulas $t(\bar x) = s(\bar y)$, where $t$ and $s$ are terms, and which agrees with $\varphi$ modulo $p(\bar x) \wedge q(\bar y)$;
  \item The formula $\varphi(\bar x)$ is stable on $p(\bar x) \wedge q(\bar y)$.
  \end{enumerate}

  Let $G =\Aut(\mcB)$. (By Stone duality, $G$ is the homeomorphism group of the space $X$ of ultrafilters on $\mcB$, i.e., $G$ is isomorphic to the homeomorphism group of the Cantor space.) Then $G$ is Roelcke precompact (the action $G \actson \mcB$ is oligomorphic) and $W(G)$ consists of all isomorphisms between Boolean subalgebras of $\mcB$. Dually, $W(G)$ is the semigroup of homeomorphisms between zero-dimensional factors of $X$. For a description of $R(\Aut(\mcB))$, see \cite{Uspenskii2001}.
\end{exm}
\begin{proof}
  Any formula of the form $t(\bar x) = s(\bar y)$ is clearly stable, so we prove the converse.
  Replacing a realisation of $p$ with the atoms of the algebra it generates, we may assume that $p(\bar x)$ and $q(\bar y)$ require that $\bar x$ and $\bar y$ be partitions of $1$, of lengths $m$ and $n$, respectively.
  Given a pair $\bar a,\bar b$ of such partitions, its type is determined by the collection of pairs $(i,j) \in m \times n$ such that $a_i \wedge b_j \neq 0$ (by quantifier elimination).

  Assume first that $n = m = 2$, that $a_i \wedge b_j = 0$ only for $i = j = 0$, and moreover, that if we invert the truth value of $a_0 \wedge b_0 = 0$, we also change the truth value of $\varphi(\bar a,\bar b)$.
  It is straightforward to construct a sequence $(\bar a^k,\bar b^k)_{k \in \bN}$ where each $\bar a^k$ and $\bar b^k$ is a partition of $1$ into two atoms, such that $a_i^k \wedge b_j^\ell = 0$ if and only if $i = j = 0$ and $k \leq \ell$, showing that $\varphi$ is unstable.

  Now Let $m,n \geq 2$, but assume that for $i,j \in \{0,1\}$ we have the same hypotheses as in the previous case, and moreover $a_0 \vee a_1 = b_0 \vee b_1$.
  Then this readily reduces to the previous case.

  Now drop the hypothesis that $a_0 \vee a_1 = b_0 \vee b_1$, keeping the others.
  Let $e = (a_0 \vee a_1) \wedge (b_0 \vee b_1)$.
  Replacing $a_0$, $a_1$, $b_0$ and $b_1$ with their respective intersections with $e$, adding the complements (if non-empty) to the list of atoms, and modifying $\varphi$ accordingly, we reduce to the previous case.

  For the general case, define a \emph{block} of $\bar a,\bar b$ as a pair $(I,J) \neq (\emptyset,\emptyset)$, where $I \subseteq m$, $J \subseteq n$ and $a_i \wedge b_j = 0$ for all $(i,j) \in \bigl[I \times (n \setminus J)\bigr] \cup \bigl[(m \setminus I) \times J\bigr]$.
  We shall say that $(i,j)$ \emph{belongs} to a block $(I,J)$ if $(i,j) \in I \times J$.
  The collection of minimal (with respect to inclusion) blocks $(I_k,J_k)$ gives rise to two respective partitions of $m$ and $n$, and determines the set of formulas $t(\bar x) = s(\bar y)$ satisfied by $\bar a,\bar b$.
  By a simple modification of $\bar a,\bar b$ we shall mean switching from $a_i \wedge b_j = 0$ to $a_i \wedge b_j \neq 0$ for a single pair $(i,j)$, which moreover belongs to some minimal block.
  Notice that this is not reversible (by a simple modification), and that it keeps the collection of blocks unchanged.

  If $\varphi(\bar x, \bar y)$ is not equivalent on $p(\bar x) \wedge q(\bar y)$ to a Boolean combination of formulas $t(\bar x) = s(\bar y)$, then there exist pairs $\bar a,\bar b$ and $\bar c,\bar d$ which have the same blocks, on which $\varphi$ differs.
  Applying to $\bar a,\bar b$ and to $\bar c,\bar d$ as many simple modifications as possible that do not change the truth value of $\varphi$, we may assume that any further simple modification must change it.
  Since $\bar a,\bar b$ and $\bar c,\bar d$ do not have the same type, a simple modification is possible for at least one of the two pairs.
  We conclude that there exists a pair $\bar a,\bar b$ such that a simple modification of the pair is possible, and any simple modification will invert the truth value of $\varphi(\bar a,\bar b)$.
  We claim that in this case there exists a minimal block $(I,J)$ and $i,i' \in I$, $j,j' \in J$, such that exactly one of the four possible intersections of $a_i,a_{i'}$ and $b_j, b_{j'}$ is empty.
  This will reduce to the previous case.

  Indeed, since a simple modification is possible, we may assume that $a_0 \wedge b_0 = 0$ and $(0,0)$ belong to some minimal block $(I,J)$.
  Let
  \begin{gather*}
    I_0 = \{ i \in I : a_i \wedge b_0 \neq 0\},
    \qquad
    J_0 = \{ j \in J : a_0 \wedge b_j \neq 0\},
  \end{gather*}
  noticing that $I_0,J_0 \neq \emptyset$.
  If there exist $(i,j) \in I_0 \times J_0$ such that $a_i \wedge b_j \neq 0$ then $0,i \in I$ and $0,j \in J$ are as desired.
  We may therefore assume that $a_i \wedge b_j = 0$ for all $(i,j) \in I_0 \times J_0$.

  Let $I_1, J_1$ be maximal such that $I_0 \subseteq I_1 \subseteq I$ and $J_0 \subseteq J_1 \subseteq J$ and $a_i \wedge b_j = 0$ for all $(i,j) \in I_1 \times J_1$, noticing that $0 \notin I_1 \cup J_1$.
  There must exist a pair $(i,j) \in (m \setminus I_1) \times (n \setminus J_1)$ such that $a_i \wedge b_j \neq 0$, since otherwise we could decompose $(I,J)$ into smaller blocks.
  Fix such a pair.
  Now there must exist $j' \in J_1$ such that $a_i \wedge b_{j'} \neq 0$, since otherwise we could add $i$ to $I_1$, contradicting maximality.
  Similarly, there exists $i' \in I_1$ such that $a_{i'} \wedge b_j \neq 0$.
  Then $i,i' \in I$ and $j,j' \in J$ are as desired, and we have proved our claim, concluding the proof.
\end{proof}

In all of the three examples above, it is easy to check from our description of the semigroup $W(G)$ that it has only two central idempotents, the identity and $0$, the empty isomorphism (or, the identity on the subalgebra $\set{0, 1}$ in the case of $\mcB$). Now Proposition~\ref{p:corr-quot} tells us that any homomorphism from $G$ to the isometry group of a reflexive Banach space has a closed image. (In fact, the group $\Homeo(2^\N)$ is totally minimal (Gamarnik~\cite{Gamarnik1991}), which is a stronger result, but in the other two cases, this seems to be new.)

Next, we consider the bounded Urysohn space $\bU_1$; let $G = \Iso(\bU_1)$. Recall from Example~\ref{ex:IsoU1} that the Roelcke compactification $R(G)$ can be identified with the space of bi-Kat\v{e}tov functions on $\bU_1 \times \bU_1$ bounded by $1$. The following calculation gives another proof of Pestov's result \cite{Pestov2007} that $W(G)$ is trivial, not using the fact that $\Iso(\bU_1)$ is a universal Polish group.
\begin{exm}
  \label{exm:StablePartUrysohn}
  \textit{The bounded Urysohn space.} Let $T = \Th(\bU_1)$. Note that, by homogeneity, $T$ eliminates quantifiers. Let $p(\bar x), q(\bar y) \in \tS_n(T)$ be complete $n$-types, and $\varphi(\bar x,\bar y)$ a formula with two groups of $n$ variables. Then the following are equivalent:
  \begin{enumerate}
  \item The formula $\varphi(\bar x)$ is constant on $p(\bar x) \wedge q(\bar y)$.
  \item The formula $\varphi(\bar x)$ is stable on $p(\bar x) \wedge q(\bar y)$.
  \end{enumerate}

  From this, we conclude that $W(G)$ is trivial.
\end{exm}
\begin{proof}
  If $X_1$ and $X_2$ are metric spaces, denote by $K_1(X_1, X_2)$ the space of all bi-Kat\v{e}tov functions on $X_1 \times X_2$ bounded by $1$.

  Let $p$ be a fixed type and let $X = \{a_0,\ldots a_{n-1}\}$ and $Y = \{b_0, \ldots, b_{n-1}\}$ be finite metric spaces whose enumerations realise $p$.
  Let $r_0 = \min \bigl\{ d(a_i,a_j), d(b_i,b_j) : i < j \bigr\}$, and for $f \in K_1(X,Y)$ let $r(f) = \min \{r_0, \min f\}$.
  Choose some pair $(i,j)$ such that $f(a_i,b_j) = \min f$ (say the first in the lexicographic ordering), and define $f'$ to agree with $f$ everywhere except for $f'(a_i,b_j) = \bigl( f(a_i,b_j) + r(f) \bigr) \wedge 1$.
  Since $r(f) \leq r_0$, a straightforward verification yields that $f' \in K_1(X,Y)$ as well; we will say that $f'$ is a \emph{simple increment} of $f$.

  By the discussion above and with quantifier elimination, we can identify the restriction of $\varphi$ to $p(\bar x) \wedge q(\bar y)$ with a (continuous) function $\overline \varphi\colon K_1(X,Y) \rightarrow \bR$, and we need to show that if $\overline \varphi$ is not constant then $\varphi$ is unstable.
  Since the constant function $1$ always belongs to $K_1(X,Y)$, we may assume that there exists $f \in K_1(X,Y)$ such that $\overline \varphi(f) \neq \overline \varphi(1)$.
  For any $\eps > 0$ we have $(f + \eps) \wedge 1 \in K_1(X,Y)$ as well, so by continuity of $\overline \varphi$ we may assume that $r(f) > 0$.
  Since $r(f') \geq r(f)$, applying a finite number of simple increments to $f$ we obtain the constant function $1$.
  We may therefore assume that there exists $f \in K_1(X,Y)$ such that $\overline \varphi(f) \neq \overline \varphi(f')$, and moreover that $(a_0,b_0)$ is the pair to which the simple increment applies.

  We now proceed as in Example~\ref{exm:StablePartRandomGraph}, letting $d(a_0^k,a_0^\ell) = d(b_0^k,b_0^\ell) = r(f)$, $d(a_0^k,b_0^\ell) = f(a_0,b_0)$ for $k < \ell$ and $d(a_0^k,b_0^\ell) = f'(a_0,b_0)$ otherwise.
  The triangle inequality holds, so this construction is legitimate, and the conclusion is as before.
\end{proof}

\end{document}